\documentclass[12pt]{amsart}

 \usepackage{graphicx}
 \usepackage{tikz-cd}
 \usepackage{ifthen}    
 \usepackage{amssymb}   
 \usepackage{amstext}
 \usepackage{amsfonts}
 \usepackage{amsmath}   
 \usepackage{latexsym}  
 \usepackage{color}
 \usepackage{epsfig}
 \usepackage{pstricks, pst-node,pst-text,pst-tree}
 \usepackage[all]{xy}
 \usepackage{array}
 \usepackage{amsthm}
 \usepackage{eucal, mathrsfs} 
 \usepackage{hyperref}

\newcommand{\numberset}[1]{\ensuremath{\mathbb{#1}}}    
\newcommand{\C}{\numberset{C}}  
\newcommand{\R}{\numberset{R}}  
\newcommand{\Z}{\numberset{Z}}  
\newcommand{\PP}{\numberset{P}}  
\newcommand{\RP}{\numberset{R}\numberset{P}}

\newcommand{\inn}[2]{ \langle {#1}, {#2} \rangle}
\newcommand{\DF}{\mathcal D^bFuk}
\newcommand{\DC}{\mathcal D^bCoh}

\newcommand{\h}{\mathbf{h}}
\newcommand{\y}{\mathbf{y}}
\newcommand{\x}{\mathbf{x}}
\newcommand{\gb}{\mathbf{g}}

\theoremstyle{definition}

\newtheorem{thm}{Theorem}[section]
\newtheorem{prop}[thm]{Proposition}
\newtheorem{lem}[thm]{Lemma}
\newtheorem{cor}[thm]{Corollary}
\newtheorem{rem}[thm]{Remark}
\newtheorem{ex}[thm]{Example}
\newtheorem{defi}[thm]{Definition}

\DeclareMathOperator{\conv}{Conv} 
 
\DeclareMathOperator{\cone}{Cone} 
\DeclareMathOperator{\vol}{vol}

\DeclareMathOperator{\spn}{span} 
\DeclareMathOperator{\Hom}{Hom}

\DeclareMathOperator{\hes}{Hess}

\DeclareMathOperator{\Log}{Log}

\DeclareMathOperator{\inter}{Int}

\newboolean{mynotes}\setboolean{mynotes}{false}

\newcommand{\mycomments}[1]{
           \ifthenelse{\boolean{mynotes}}
                      {#1}{}
           }
	   
\newcommand{\marginlabel}[1]
{\mbox{}\marginpar[\raggedleft $\longrightarrow$ \\ \tiny\sf #1]{\raggedright $\longleftarrow$\\ \tiny\sf #1}}

\newcommand{\todo}[1]{\mycomments{\marginlabel{\tiny\sf{#1}}}}

\begin{document}

\title{Lagrangian pairs of pants}
\author{Diego Matessi}

\begin{abstract}
We construct a Lagrangian submanifold, inside the cotangent bundle of a real torus, which we call a Lagrangian pair of pants. It is given as the graph of the differential of a smooth function defined on the real blow up of a Lagrangian coamoeba. Lagrangian pairs of pants are the main building blocks in a construction of smooth Lagrangian submanifolds of $(\C^*)^n$ which lift tropical subvarieties in $\mathbb R^n$. As an example we explain how to lift tropical curves in $\R^2$ to Lagrangian submanifolds of $(\C^*)^2$. We also give several new examples of Lagrangian submanifolds inside toric varieties, some of which are monotone. 
\end{abstract}

\maketitle


\section{Introduction}
\subsection{Summary of main results.}
Applications of tropical geometry to problems in complex geometry are now abundant in the literature, starting from the pioneering work of Mikhalkin, who proved the first ``tropical to complex'' correspondence theorem \cite{Mikh-tropical} where the enumeration of curves in $\PP^2$ was proved to be equivalent to the enumeration of tropical curves in $\R^2$.  One way to think about such correspondence theorems is that a tropical hypersurface $\Xi$ in $\R^n$ encodes the information to produce a one parameter family of algebraic hypersurfaces $Y_t$ in $(\C^*)^n$, which in some sense ``converge'' to the tropical hypersurface. More precisely, if one considers the one parameter family of maps $\Log_t: (\C^*)^n \rightarrow \R^n$ given by $\Log_t(z_1, \ldots, z_n) = (\log_t |z_1|, \ldots, \log_t|z_n|)$, then the amoebas of $Y_t$, i.e. the sets $\mathcal A_t = \Log_t(Y_t)$, converge to the tropical curve. Indeed Mikhalkin \cite{mikh_pants}  proves the stronger result that $Y_t$, in some sense, converges to a complexified tropical hypersurface (also called the phase tropical hypersurface). This is a piecewise linear lift $\hat \Xi \subset (\C^*)^n$ of $\Xi$, constructed by adding some fibres over $\Xi$ (i.e. coamebas). Kerr and Zharkov \cite{kerr_zha_phase_trop} and Kim and Nisse \cite{kim_nisse_phase_trop} have shown that $\hat \Xi$ is a topological manifold homeomorphic to $Y_t$. 

In this paper and in the forthcoming one \cite{trop_hyp_to_lag} we take inspiration from the philosophy of mirror symmetry that symplectic geometry is mirror to complex geometry to prove a symplectic version of the above result. Namely, given a tropical hypersurface $\Xi$ in $\R^n$ we construct a Lagrangian PL lift $\hat \Xi$ in $(\C^*)^n$ (different from the one in complex geometry, but similar) and we prove the following:
\begin{thm} \label{main_thm} Given a smooth tropical hypersurface $\Xi$ in $\R^2$ or $\R^3$, there is a one parameter family of smooth Lagrangian submanifolds $\mathcal L_t$ of respectively $(\C^*)^2$ or $(\C^*)^3$ such that $\mathcal L_t$ is homeomorphic to the PL lift $\hat \Xi$ of $\Xi$ and converges to it in the Hausdorff topology as $t \rightarrow 0$. 
\end{thm}
The main building blocks in the proof of this theorem are certain Lagrangian submanifolds inside the cotangent bundle of a real torus, which we call Lagrangian pairs of pants. These are the Lagrangian lifts of tropical hyperplanes. After introducing these objects, in this paper we prove the theorem in the case of tropical curves in $\R^2$. We will also give some examples of Lagrangian lifts of non-smooth tropical curves and discuss several examples of Lagrangian submanifolds in toric varieties, including two examples of monotone Lagrangian tori in $\PP^2$ and in $\PP^1 \times \PP^1$.  As we were writing this paper we learned that Mikhalkin also has a similar construction for the case of tropical curves in  $\R^n$. In discussions we had with him he told us about his forthcoming paper \cite{mikh_trop_to_lag}, where he will also give some interesting applications to enumerative problems of special Lagrangian submanifolds. He also told us a very nice way to use Lagrangian lifts of tropical curves to construct examples of non-orientable Lagrangian surfaces in $\C^2$, see \S  \ref{non_orientable_ex} for a sketch of this idea.

The restriction to smooth tropical hypersurfaces in $\R^2$ or $\R^3$ is due to technical difficulties (analytic and combinatorial) which we could not for the moment overcome, but we are convinced that the statement should be provable for general tropical subvarieties of any codimension using substantially the same method. Since the proof of Theorem \ref{main_thm} for the $3$-dimensional case is longer and technically more difficult we decided to write it in a separate paper \cite{trop_hyp_to_lag}. 

\subsection{Lagrangian pairs of pants} Given a real, $n$-dimensional torus $T$, a Lagrangian coamoeba inside $T$ is a subset $C$ which topologically is given by two copies of an $n$-dimensional simplex glued together at the vertices (see Figure \ref{co_am}). The name is inspired by the objects with the same name in the context of complex tropical geometry (\cite{nisse_sottile_phase_limit}, \cite{nisse_sottile_nonAch_coam}). Lagrangian coamoebas are naturally dual to tropical hyperplanes, in particular there is an inclusion reversing duality between faces of $C$ and cones of the tropical hyperplane.  We find a smooth function $F: C \rightarrow \R$ such that the graph of its differential in the cotangent bundle $T^*T$ of $T$ extends to a smooth embedding of the real blow up of $C$ at its vertices. A Lagrangian pair of pants is the image $L$ of this embedding. Now observe that $T^*T$ is symplectomorphic to $(\C^*)^n$, thus we have that $L$ is a Lagrangian submanifold of $(\C^*)^n$. If we project to $\R^n$ via the $\Log$ map, the image of $L$ resembles the amoeba of a complex hyperplane (see Figure \ref{lagrangian_amoeba}).

\subsection{Lagrangian PL lifts of tropical hypersurfaces}
Roughly speaking a tropical hypersurface $\Xi$ in $\R^n$ is a union of $(n-1)$-dimensional rational polyhedra glued along faces. To each polyhedron we attach a weight, i.e. a positive integer,  and we require a balancing condition at the intersections of these polyhedra. The most basic examples are the tropical hyperplanes, these consist of the non-smooth locus of the function $\min(0, x_1, \ldots, x_n)$. The smoothness condition in the statement of Theorem \ref{main_thm} consists in requiring that an $(n-2)$-dimensional face is the intersection of exactly three polyhedra and that locally near this face $\Xi$ looks like a tropical line times the face. The polyhedra, in the smooth case, are unweighted. A Lagrangian PL lift of $\Xi$ is a subset $\hat \Xi$ of $(\C^*)^n$ such that the following diagram commutes
\begin{equation} \label{diagr_lift}
    \begin{tikzcd}
         \hat \Xi  \arrow[d]  \arrow[r] & (\C^*)^n  \arrow[d]\\
         \Xi \arrow[r] & \R^n
    \end{tikzcd}
\end{equation}
where the horizontal arrows are the inclusions and the vertical one is the $\Log$ map.  Let us view  $(\C^*)^n$ as $\R^n \times (\R^n / \Z^n)$. Given a point $b$ in the interior of an $(n-1)$-dimensional face of $\Xi$, the fibre over $b$ in $\hat \Xi$ is the circle spanned by the orthogonal complement to the face. More generally if $b$ is a point in the interior of a $k$-dimensional face of $\Xi$ then the fibre is an $(n-k)$ dimensional coamoeba inside the sub-torus spanned by the orthogonal complement of the face. This ensures that $\hat \Xi$ is a topological submanifold which is Lagrangian at smooth points.  For instance, in the case of a tropical hyperplane $\Gamma$, its $k$-dimensional faces (or cones) are labeled by $\Gamma_J$, where $J$ is a multi index of length $k$. The $n-k$-dimensional face of the Lagrangian coamoeba $C$ which is dual to $\Gamma_J$ is labeled by $E_J$. We have that the lift of $\Gamma_J$ is
\[ \hat \Gamma_J = \Gamma_J \times E_J \]
and the lift of $\Gamma$ is the union of the lifts of all of its cones. 

\subsection{Sketch of the proof of Theorem \ref{main_thm}}
One of the interesting properties of Lagrangian pairs of pants is that the image of a neighborhood of a face $E_J$ is also the graph of the differential of another function defined on the lift $\hat \Gamma_J$ of $\Gamma_J$. This function is a Legendre transform of $F$ (see Corollary \ref{fbr_bndl}). This is the key idea which allows us to glue together Lagrangian pairs of pants to build the family $\mathcal L_t$ of Theorem \ref{main_thm}. In fact if $\Xi$ is a smooth tropical hypersurface, every vertex locally looks like a tropical hyperplane, thus we can locally replace the PL lift with a Lagrangian pair of pants. Then we need to glue together Lagrangian pairs of pants over vertices which lie on the same face. This can be done by perturbing the functions defined on the lift of this face (i.e. the Legendre transforms of $F$). This program is particularly simple in the case of tropical curves in $\R^2$ (see Section \ref{lifting_trop}). 

\subsection{Generalizations and examples} In Sections \ref{generalizations} and \ref{intoric} we give various examples and discuss possible generalizations of the construction. For instance we sketch how to lift tropical curves in higher codimensions, we give some examples of lifts of non-smooth tropical curves and we show how the same tropical curve admits various Lagrangian lifts obtained by twisting a given lift by local sections. Finally we give various examples of Lagrangian submanifolds of toric varieties.  The idea to construct surfaces inside toric and almost toric varieties by lifting curves in the moment polytope can also be found in the work of Symington, see Definition 7.3 and Theorem 7.4 of \cite{sym_at}.  The complement of the toric boundary of a smooth toric variety is symplectomorphic to $\Delta^o \times T$, where $\Delta^o$ is the interior of a Delzant polytope $\Delta \subset \R^2$ and $T$ is the $2$-torus. Thus we can lift a tropical curve $\Xi$ in $\Delta$ to a Lagrangian submanifold of $\Delta^o \times T$ and then take its closure $\mathcal L$ in the toric variety $X_{\Delta}$. The topology of $\mathcal L$ depends on how the edges of $\Xi$ hit the boundary of $\Delta$. For instance $\mathcal L$ may or may not have boundary. In particular, if the edges of $\Xi$ end on the vertices of $\Delta$ and ``bisect'' the angle at these vertices, then $\mathcal L$ is smooth without boundary. For example the lifts of the tropical curves in Figures \ref{co_am_torus_p2} and  \ref{co_am_torus_p1p1} are Lagrangian tori respectively in $\PP^2$ and $\PP^1 \times \PP^1$. One interesting case, which was pointed out to us by Mikhalkin (see also his forthcoming article \cite{mikh_trop_to_lag}), is when the edges hit the boundary with ``multiplicity two'', in which case taking the closure has the effect of gluing in a  M\"obius strip. This produces examples of non-orientable surfaces, see Figure \ref{non_orientable} for an example of a tropical curve which lifts to a Lagrangian surface in $\C^2$ with Euler characteristic $-4$. 

We can also construct two interesting examples of monotone Lagrangian tori in $\PP^2$ and in $\PP^1 \times \PP^1$. These are the lifts respectively of the tropical curves in Figure \ref{co_am_genus1_p2} and \ref{monotone_p1p1}. We do not know whether these examples are Hamiltonian isotopic to other known monotone tori (\cite{chekanov_schlenk_tori}, \cite{vianna_infte_montone_tori}, \cite{abreu_gadbled_monotone_lag}). 

\subsection{Mirror symmetry}  Let us view the results in this article in the context of mirror symmetry, which is one of our main motivation for this work. The homological mirror symmetry conjecture states that given two mirror Calabi-Yau manifolds $X$ and $\check X$ then the derived Fukaya category $\DF(\check X)$ of $\check X$ should be equivalent to the derived category of coherent sheaves $\DC(X)$ of $X$. 
Objects in the former category are Lagrangian submanifolds with vanishing Maslov class. In the Strominger-Yau-Zaslow (SYZ) interpretation of mirror symmetry, two mirror Calabi-Yau manifolds $X$ and $\check X$ should come with dual special Lagrangian torus fibrations $f: X \rightarrow B$ and $\check f: \check X \rightarrow B$. 
These fibrations in general have singular fibres, so let $B_0$ be the locus of smooth fibres. The Lagrangian condition implies that $B_0$ carries an integral affine structure, thus locally $B$ looks like $M_{\R} = M \otimes \R$, where $M$ is an $n$-dimensional lattice (i.e. $M \cong \Z^n$). If we view $X$ as a complex manifold and $\check X$ as a symplectic one, the duality condition on the two fibrations imply that locally $X$ looks like $M_{\R} \times M_{\R}/M$ and $\check X$ looks like $M_{\R} \times N_{\R}/N$, where $N = \Hom(N, \Z)$.  
The integral affine structure makes of $B$ an ambient space where we can define tropical subvarieties.  So let $\Xi$ be a $d$-dimensional tropical subvariety of $B$ with boundary on the discriminant locus $\Delta = B-B_0$ of the fibration. 
Now $X$ is the natural ambient space where to define a complex $PL$ lift of $\Xi$, i.e. where the fibres over $d$-dimensional faces in a diagram such as \eqref{diagr_lift}, are the tangent spaces. With suitable assumptions on how $\Xi$ interacts with the discriminant locus, the complex $PL$ lift of $\Xi$ should be a $2p$-dimensional topological submanifold, or a $(d,d)$-cycle. 
On the other hand $\check X$ is the natural ambient space where to define the Lagrangian $PL$ lifts of $\Xi$, where the fibres are the orthogonal spaces to the tangent spaces. We view such lifts as $(d, n-d)$-cycles.
 If we naively consider the complex and Lagrangian PL lifts as actual complex and Lagrangian subvarieties, we could formulate the homological mirror symmetry conjecture by stating that different Lagrangian lifts of the same tropical subvariety should correspond to coherent sheaves supported on the complex lift. 

Although this statement is just a very rough approximation, aspects of this program have already been carried out in the literature to various degrees of precision and depth. For instance Section 6.3 of \cite{diri_branes} discusses this correspondence in the case $X$ is endowed with the semiflat complex structure and the tropical varieties are integral affine submanifolds. In this case there are no singular fibres and the complex and Lagrangian $PL$ lifts are actually complex and Lagrangian smooth submanifolds. In \cite{onHmstoric}, together with Gross we propose an explicit correspondence between certain Lagrangian and complex lifts in the case of toric Calabi-Yau manifolds and we find some evidence that this correspondence should induce an equivalence of derived categories. In this case the underlying tropical subvarieties are just disks with boundary on the discriminant locus. Similar results are proved in \cite{chan:ueda} and \cite{CPU:hms:conifold}. 

A deeper justification for naming the above lifts $(d,d)$-cycles or $(d, n-d)$-cycles comes from so called tropical homology. Indeed Itenberg, Katzarkov, Mikhalkin and Zharkov \cite{trop_hom} consider similar cycles which define certain $(p,q)$-homology groups of a tropical variety and show that the dimensions of these groups match the Hodge numbers of the corresponding complex variety. Similar cycles and homology groups have been defined in \cite{can_coord_Sieb_Rud}. 

One major difficulty in this program is to prove tropical to complex correspondence theorems in the Calabi-Yau setting and for tropical subvarieties of arbitrary dimension. So far such results have been proved for curves or hypersurfaces inside toric varieties \cite{Mikh-tropical}, \cite{mikh_pants}, \cite{NiSi}. One of the goals of the Gross-Siebert program is to prove tropical to complex correspondence theorems in the case of toric degenerations of families of Calabi-Yau varieties via logarithmic geometry (\cite{GrSi_re_aff_cx}, \cite{GS_log_GW}).  Constructing smooth Lagrangian lifts of tropical subvarieties in the Calabi-Yau case should be easier. Indeed together with Casta\~no-Bernard we have constructed Lagrangian fibrations on a large class of symplectic Calabi-Yau manifolds \cite{CB-M}. Therefore it should be possible to construct smooth Lagrangian lifts of tropical subvarieties in the same way as done here, but with an additional analysis of the interactions with the singular fibres (see for instance the constructions in \cite{onHmstoric} or, at a topological level, in \cite{CB-M-Con}). 
\subsection{Structure of the paper} In Section \ref{trop_geom} we give some background on tropical geometry, we introduce Lagrangian coamoebas and define the Lagrangian PL lifts of tropical hypersurfaces. In Section \ref{LagPants} we define the Lagrangian pair of pants and prove many of its properties. In Section \ref{lifting_trop} we prove Theorem \ref{main_thm} for the case of tropical curves in $\R^2$. In Section \ref{generalizations} we discuss various generalizations, such as the construction of different lifts of the same tropical curve, examples of lifts of non-smooth tropical curves and the case of curves in $\R^n$. In Section \ref{intoric} we give various examples of Lagrangian lifts inside toric varieties. In the Appendix we prove the technical result showing that the Hessian of the function used to define the Lagrangian pair of pants is negative definite. 

\subsection{Notation.} \label{notation}Throughout the paper we will use the following notations. Given a set of vectors $u_1, \ldots, u_k$ in a vectors space $V$, the cone generated by these vectors is the set
\[ \cone \{ u_1, \ldots, u_k \} = \left \{ \sum_{j=1}^{k} t_j u_j \, | \, t_j \in \R_{\geq 0} \right \}. \]
Given a subset $A$ of an affine space, we will denote the convex hull of $A$ by
\[ \conv A. \]
Given a subset $W$ of an affine space, the notation 
\[ \inter W \]
stands for the relative interior of $W$. Namely, we consider the smallest affine subspace containing $W$, then $\inter W$ will be the topological interior relative to this affine subspace. This for examples applies to faces of polyhedra or cones. 

\subsection*{Acknowledgments} 
I wish to thank Mark Gross for suggesting this problem, Ricardo Casta\~no-Bernard for many discussions and Grigory Mikhalkin for sharing many of his ideas on tropical to Lagrangian correspondences. I was partially supported by the grant FIRB 2012 ``Moduli spaces and their applications'' and by the national research project ``Geometria delle variet\`a proiettive'' PRIN 2010-11. I am a member of the INDAM research group GNSAGA.

\section{Tropical hypersurfaces and their piecewise linear lifts} \label{trop_geom}
\subsection{The set-up} \label{setup} Let $M \cong \Z^{n+1}$ be a lattice of rank $n+1$ and let $N = \Hom(M, \Z)$ be its dual lattice. We define $M_{\R} := M \otimes_{\Z} \R$ and similarly $N_{\R}$. Since $M_{\R}$ is the dual of $N_{\R}$ the space $M_{\R} \oplus N_{\R}$ has a natural symplectic form
\[ \omega( m \oplus n, m' \oplus n') = \inn{m}{n'} - \inn{m'}{n} \]
where $\inn{\cdot}{\cdot}$ is the duality pairing. We will consider the $n+1$-dimensional torus
\[ T = N_{\R} / N \]
whose cotangent bundle is 
\[ T^*T = M_{\R} \times N_{\R}/N. \]
Then $\omega$ is the standard symplectic form on $T^*T$. The projection 
\[f: T^{*}T \rightarrow M_{\R} \]
 is a Lagrangian torus fibration. The goal of this section is to define tropical hypersurfaces $\Xi$ in $M_{\R}$ and their piecewise linear Lagrangian lifts $\hat \Xi$ (see \eqref{diagr_lift}). 

We will often identify $M_{\R}$ with $\R^{n+1}$ by choosing a basis $\{u_1, \ldots, u_{n+1} \}$ of $M$ and denote the corresponding coordinates in $M_{\R}$ by $x = (x_1, \ldots, x_{n+1})$. Similarly we also identify $N_{\R}$ with $\R^{n+1}$ by choosing a basis $\{ u_1^*, \ldots, u_{n+1}^* \}$ of $N_{\R}$ such that 
\begin{equation} \label{dual_basis}
   \inn{u^*_{j}}{u_k} = \frac{1}{\pi} \delta_{jk} 
\end{equation}
and denote the corresponding coordinates by $y = (y_1, \ldots, y_{n+1})$. In particular $N$ is identified with $\pi \Z^{n+1}$ and thus 
\begin{equation} \label{toro}
T = \R^{n+1}/ \pi \Z^{n+1}.
\end{equation}
We denote by $[y]$ the element of $T$ represented by $y$. 
The symplectic form $\omega$ becomes 
\begin{equation} \label{st_sympl}
       \omega = \frac{1}{\pi}\sum_{i=1}^{n+1} dx_i \wedge dy_i.
\end{equation}
We also have that $T^*T$ is symplectomorphic to $(\C^*)^n$ with the symplectic form 
\[      \omega = \frac{i}{4\pi}\sum_{k=1}^{n+1} \frac {dz_k \wedge d \bar z_k}{|z_k|^2}, \]
via the symplectomorphism $(x_k,y_k) \mapsto z_k = e^{x_k+i2y_k}$.

We will also consider the following set of symplectic automorphisms of $T^*T$.

\begin{defi} \label{automorph} Let $A: N_{\R} \rightarrow N_{\R}$ be an integral linear automorphism, i.e. a linear automorphism such that $A(N) =N$, and let $x_0 \in M_{\R}$ and $[y_0] \in T$ be two points. An affine automorphism of $T$ is a map $\rho: T \rightarrow T$ of the type
\[ \rho([y]) = \left[ y_0+ Ay \right]. \]
We define the affine automorphism of $M_{\R}$ associated to $\rho$ and $x_0$ to be the map 
\[ \rho^*_{x_0} (x)= x_0 + (A^t)^{-1}x, \]
where $A^t$ is the transpose of $A$. The affine (symplectic) automorphism of $T^*T$ associated to $\rho$ and $x_0$ is the map $\hat \rho_{x_0}: T^*T \rightarrow T^*T$ given by
\[ \hat{\rho}_{x_0}(x,[y])= \left(  \rho^*_{x_0}(x), \rho([y]) \right), \]
 Clearly $\hat{\rho}_{x_0}$ takes the fibre $f^{-1}(0)$ to the fibre $f^{-1}(x_0)$. 
\end{defi}

\subsection{Tropical hypersurfaces} \label{trop_hyper}  A subset $P \subset N_{\R}$ is a convex lattice polytope if it is the convex hull of a finite set of points in $N$. A subdivision of $P$ in smaller lattice polytopes $P_1, \ldots, P_k$ is called {\it regular} if there exists a convex piecewise affine function $\nu: P \rightarrow \R$, such that $\nu$ is integral, i.e. $\nu(P \cap N) \subset \Z$, and the $P_i$'s coincide with the domains of linearity of $\nu$. With a slight abuse of notation the pair $(P, \nu)$ will also denote the set of simplices in the decomposition, i.e. all the $P_k$'s and all of their faces of any dimension. Therefore we will write 
$e \in (P, \nu)$ to indicate that $e$ is a simplex in the decomposition. The relation of inclusion among faces will be denoted by 
\[ f \preceq e. \]
We say that the subdivision is {\it unimodal} if all the $P_i$'s are elementary simplices.  

The discrete Legendre transform of $\nu$ is defined to be the following function $\check{\nu}: M_{\R} \rightarrow \R$:
\begin{equation} \label{cknu}
\check{\nu}(m) = \min \{ \inn{v}{m} + \nu(v), \ v \in P \cap N\}. 
\end{equation}

Also $\check \nu$ gives a decomposition of $M_{\R}$ in the convex polyhedra given by the domains of linearity of $\nu$. As above, the pair  $(M_{\R}, \check \nu)$ will also denote the set of all polyhedra in the subdivision and their faces. 

\begin{defi} The {\it tropical hypersurface} associated to the pair $(P, \nu)$ is the subset $\Xi \subset M_{\R}$ given by the points where $\check{\nu}$ fails to be smooth, i.e. it is the union of polyhedra of $(M_{\R}, \check \nu)$ of dimension at most $n$. We say that $\Xi$ is {\it smooth} if the subdivision of $P$ induced by $\nu$ is unimodal.
\end{defi}

\begin{ex} Let $M= \Z^{n+1}$ and let $P$ be the standard simplex (i.e. the convex hull of the origin and the canonical basis of $\Z^{n+1}$). Let $\nu$ be the zero function. Then 
\[ \check{\nu} = \min \{0, x_1, \ldots, x_{n+1} \}. \] 
The corresponding tropical hypersurface is called the standard tropical hyperplane (see next section). 
\end{ex}

\begin{figure}[!ht] 
\begin{center}
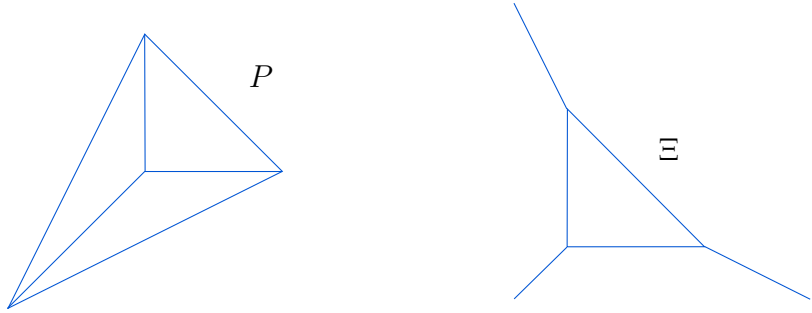
\caption{The polytope $P$ and the tropical curve $\Xi$.} \label{trop_hype}
\end{center}
\end{figure}
\begin{ex} \label{triangle_ex}
Let $P = \conv \{ (0,0), (1,2), (2,1) \}$ with its unique unimodal subdivision (see Figure \ref{trop_hype}). It is induced by a piecewise affine function $\nu$ such that $\nu(0,0) = 1$ and $\nu(2,1) = \nu(1,2) = \nu(1,1) = 0$. Then we have 
\[ \check \nu(x_1, x_2) = \min \{ 1, x_1+x_2 , 2x_1+ x_2, x_1+ 2x_2 \} \]
and the corresponding tropical curve $\Xi$ is as in Figure~\ref{trop_hype}.   
\end{ex}

The subdivision $(M_{\R}, \check \nu)$ is dual to the subdivision $(P, \nu)$. The $n+1$ dimensional polyhedra of $(M_{\R}, \check \nu)$ are the closures of the connected components of the complement of $\Xi$. These are in bijection with the vertices of $(P, \nu)$. Let us denote this bijection by $v \mapsto \check v$ for all $v \in P$. The polyhedron $\check v$ is the region where the minimum in $\eqref{cknu}$ is attained by the affine function $m \mapsto \inn{v}{m} + \nu(v)$. Similarly there is a bijection between $k$-dimensional faces of $(P, \nu)$ and $(n+1)-k$ dimensional faces of $(M_{\R}, \check \nu)$, which we denote by 
\[ e \mapsto \check e. \]
 We have that $\check e$ is the locus where the minimum in \eqref{cknu} is attained by the affine functions corresponding to the vertices of $e$. In particular $n$-dimensional faces of $\Xi$ correspond to edges in $(P, \nu)$. If $v_1$ and $v_2$ are the vertices of $e$, $\check e$ lies in a hyperplane orthogonal to $v_1 - v_2$.

\subsection{The tropical hyperplane} \label{stTrop} Let $\{u_1, \ldots, u_{n+1} \}$ be a basis of $M$ inducing coordinates $x=(x_1, \ldots, x_{n+1})$ on $M_{\R}$. As already mentioned in the previous section, the {\it standard tropical hyperplane} in $M_{\R}$ is the tropical hypersurface $\Gamma$ given by the non-smooth locus of the function $\min \{0, x_1, \ldots, x_{n+1} \}$. It can be described as the union of the following cones. Let 
\begin{equation} \label{u0}
     u_0 = - \sum_{j=1}^{n+1} u_j.
\end{equation}
Given a proper subset $J \subsetneq \{0, \ldots, n+1 \}$ let $|J|$ be its cardinality and let 
\[ \Gamma_{J} = \cone \{ u_j, \, j \in J \}. \] \todo{explain notation cone?}
For convenience let us also define
\[ \Gamma_{\emptyset} = \{ 0 \},\]
which is the vertex of $\Gamma$. We have that
\[ \Gamma = \bigcup_{0 \leq |J| \leq n} \Gamma_J. \]
Recall from toric geometry that the collection of all cones $\Gamma_{J}$ (including those of dimension $n+1$) forms the fan of $\PP^{n+1}$ which we denote here by $\Sigma_{\Gamma}$, i.e. 
\[ \Sigma_{\Gamma} = \{ \Gamma_J \}_{0 \leq |J| \leq n+1}. \]
Thus $\Gamma$ can be viewed as the $n$-skeleton of $\Sigma_{\Gamma}$.


\subsection{Lagrangian coamoebas} \label{stCoam}
Here we define the Lagrangian coamoeba in the torus $T = N_{\R}/N$, which is dual to the tropical hyperplane. Let $\{ u_1^*, \ldots, u_{n+1}^* \}$ be the basis of $N_{\R}$ satisfying \eqref{dual_basis}. Thus the torus $T$ is as in \eqref{toro}. 
Consider the points
\[ p_0 = 0 \quad \text{and} \quad p_k = \frac{\pi}{2}u^*_k, \  \ (k=1, \ldots, n+1). \]  
Denote by $C^+$ the set of points $[y] \in T$ which are represented either by a vertex or by an interior point of the $(n+1)$-dimensional simplex with vertices the points $p_0, \ldots, p_{n+1}$. Let $C^-$ be the image of $C^+$ with respect to the involution $[y] \mapsto [-y]$. 
The {\it (standard) $(n+1)$-dimensional Lagrangian coamoeba} is the set $C = C^+ \cup C^-$ (see Figure \ref{co_am}). Notice that this definition makes sense also when $n=0$, in which case $C = \R/ \pi \Z$. The points $[p_0], \ldots, [p_{n+1}]$ are called the {\it vertices} of $C$. 

For any subset $J  \subsetneq \{0, \ldots, n+1\}$, denote by $E^+_J$ the set of points $[y] \in T$ which are represented either by a vertex or by a point in the relative interior of the $(n+1-|J|)$-dimensional simplex with vertices the points  $\{p_k \}_{k \notin J}$. We let $E_J^-$ be the image of $E^+_J$ via the involution $[y] \mapsto [-y]$. 
We define the {\it $J$-th face} of $C$ to be the set $E_J = E_J^+ \cup E_J^-$.   Clearly $E_J$ is homeomorphic to an $(n+1-|J|)$-dimensional Lagrangian coamoeba.  If $J = \{ j \}$ then we denote $E_J$ by $E_j$ and we call it the {\it $j$-th facet} of $C$. The closures $\bar E_J$ of the $J$-th faces satisfy
\[ \begin{split}
              \bar E_j   & = \left \{ [y] \in \bar C \, | \, \inn{u_j}{y} = 0 \mod  \Z \right \}, \quad \text{when} \ j=1, \ldots, n+1 \\
              \bar E_0  & = \left \{ [y] \in \bar C \, | \, \inn{u_0}{y} = \frac{\pi}{2} \mod \Z \right \}, \\
              \bar E_{J} & = \bigcap_{j \in J} \bar E_j.
   \end{split}
                     \] 
where $u_0$ is the vector defined in \eqref{u0}. For convenience we also define 
\[ E_{\emptyset} = C. \]
Faces of dimension $1$ (i.e. when $|J| = n$) are called edges. Notice that if we denote by $J_k$ the complement of $k$ in $\{0, \ldots, n+1 \}$, then 
\[ p_k = E_{J_k}. \]

\begin{figure}[!ht] 
\begin{center}
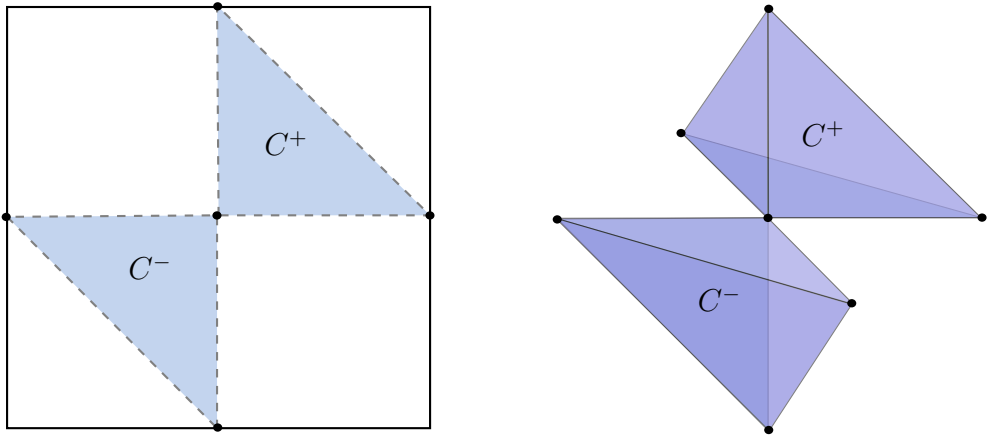
\caption{The $2$ and $3$ dimensional standard coamoebas. Notice that they contain their vertices but they do not contain any of their higher dimensional faces.} \label{co_am}
\end{center}
\end{figure}

There is a one to one inclusion reversing correspondence between faces of $C$ and cones of $\Sigma_{\Gamma}$ where $E_J$ corresponds to $\Gamma_J$. For instance vertices of $C$ correspond to $n+1$-dimensional cones. 

\subsection{The Lagrangian PL-lift of $\Gamma$} \label{plift} For every subset $J \subset \{ 0, \ldots, n+1 \}$ with $0 \leq |J| \leq n$, consider the following $n+1$ dimensional subsets of $M_{\R} \times T$:
\begin{equation} \label{ends_lift}
         \hat{\Gamma}_{J} = \Gamma_J \times E_{J}.
\end{equation}
The piecewise linear lift (or PL-lift) of $\Gamma$ is defined to be 
\begin{equation} \label{plLift}
  \hat \Gamma = \bigcup_{0 \leq |J| \leq n} \hat{\Gamma}_J.
\end{equation}
Notice that when $n=1$, then $\hat \Gamma$ is homeomorphic to $S^2$ with three punctures, i.e. to a pair of pants. Therefore we call $\hat \Gamma$ the PL-pair of pants. 

\subsection{Symmetries of $\Gamma$ and $C$}
\label{symmetries}
For every $k=1, \ldots, n+1$ let $R_k$ be the unique affine automorphism of $T$ which maps $C^+$ to itself, exchanges the vertices $p_0$ and $p_k$ and fixes all other vertices. Define $G$ to be the group generated by the maps $R_k$. We have that $G$ acts on the coamoeba $C$. The elements $R_k$ permute the faces of $C$ according to the following rule. Let $R_k$ act on the set of indices $\{0, \ldots, n+1 \}$ as the transposition which exchanges $0$ and $k$ and extend this action to the set of subsets $J \subseteq \{0, \ldots, n+1 \}$ . \todo{exchanged $R^*$ with $R$, check consistency} Then clearly
\[ R_k E_J = E_{R_kJ}. \]

\medskip

Dually let us define the group acting on $\Gamma$. Let $R^*_k$ be the affine automorphism of $M_{\R}$ associated to $R_k$ and the origin
(see Definition \ref{automorph}). It can be seen that $R^*_k$ maps $\Gamma$ to itself. Indeed if $u_0, \ldots, u_{n+1}$ are the vectors in $M_{\R}$ as in \S \ref{stTrop}, then $R^*_k$ is the unique linear map which exchanges $u_0$ and $u_k$ and fixes all other $u_j$'s.   More explicitly
\[ R^*_k(x) = (x_1 - x_k, \, \ldots, \, x_{k-1} - x_k, \, -x_k, \, x_{k+1}-x_k, \, \ldots, \, x_{n+1}-x_k ).\]
We have that $R^*_k$ permutes  the cones of $\Sigma_{\Gamma}$ according to the rule
\begin{equation} \label{action_cones}
      R^*_k \Gamma_J = \Gamma_{R_k J}.
\end{equation}
Denote by $G^*$ the group generated by the transformations $R^*_k$. Then $G^*$ acts on $\Gamma$. It is easy to see that 
\[ R_k: y \longmapsto (R^*_k)^t(y) + p_k \]
where $(R^*_k)^t$ is the transpose of $R^*_k$. 

We can combine the actions of $G$ and $G^*$ to get an action on the PL-pair of pants $\hat \Gamma$ via the following affine symplectic automorphisms of $T^*T$:
\begin{equation} \label{glob_symm}
 \mathcal R_k(x,y) = (R^*_k x, R_k y).
\end{equation}
Let $\mathcal G$ be the group generated by the $\mathcal R_k$'s. Then $\mathcal G$ acts on $\hat \Gamma$.

\subsection{Lagrangian piecewise linear lifts of tropical hypersurfaces} \label{LagPLift} We assume now that $\Xi$ is a smooth tropical hypersurface in $M_{\R}$ given by a pair $(P, \nu)$ as in \S \ref{trop_hyper} and we define its $PL$-lift $\hat \Xi$ inside $M_{\R} \times N_{\R} / N$. Given a $k$-dimensional face $e \in (P, \nu)$, with $k=1, \ldots, n+1$, let $\check e$ be the dual $(n+1)-k$ dimensional face of $\Xi$. We will use the involution $\iota$ of $M_{\R} \times N_{\R} / N$ given by $\iota: (x, [y]) \mapsto (x, [-y])$.
Define the following subsets of $N_{\R}/N$:
\[ \begin{split}
            & \bar{C}^{+}_{e}  = \{ [y] \in N_{\R} / N \, | \, 2(y-k) \in e \ \text{for some} \ k \in N \},\\
            & \bar{C}^{-}_{e} = \iota \left(  \bar{C}^{+}_{e} \right), \\
            & \bar{C}_{e}  = \bar{C}^{+}_{e} \cup \bar{C}^{-}_{e}.
   \end{split} \]
A point $[y] \in N_{\R} / N$ is a {\it vertex} of $\bar{C}_e$ if $2(y-k)$ is a vertex of $e$ for some $k \in N$. We define $C^+_e$ (resp. $C^{-}_e$ and $C_{e}$) to be the set of points $[y]$ which are either vertices or relative interior points of $\bar{C}^+_e$ (resp. $\bar{C}^{-}_e$ and $\bar{C}_{e}$). Clearly $\bar{C}^+_e$ (resp. $\bar{C}^{-}_e$ and $\bar{C}_{e}$) is the closure of $C^+_e$ (resp. $C^{-}_e$ and $C_{e}$).

Now define the Lagrangian lift of $\check e$ to be 
\[ \hat e = \check e \times C_{e}. \]
Clearly $\hat e$ is $(n+1)$-dimensional, moreover the tangent space of $C_{e}$ is the orthogonal complement of the tangent space to $\check e$, thus the interior of $\hat e$ is a Lagrangian submanifold of $M_{\R} \times N_{\R} / N$. We define the Lagrangian $PL$-lift of $\Xi$ to be 
\[ \hat \Xi = \bigcup_{e} \, \hat e \]
where the union is over all faces in $(P, \nu)$ of dimensions $k=1, \ldots, n+1$. It can be shown that $\hat \Xi$ is an $(n+1)$-dimensional topological submanifold of $M_{\R} \times N_{\R} / N$.

\begin{ex} Let us consider $\Xi$ given in Example \ref{triangle_ex}. Then the edges of $\Xi$ correspond to the edges of $(P, \nu)$. If $\check e$ is an edge of $\Xi$, then $C_{e}$ coincides with $e$ and $\hat e$ is a cylinder. If $\check e$ is a vertex of $\Xi$, then $e$ is one of the two dimensional simplices of $(P, \nu)$. The sets $C_e$ are drawn in Figure \ref{pl_lift}
\begin{figure}[!ht] 
\begin{center}
\includegraphics{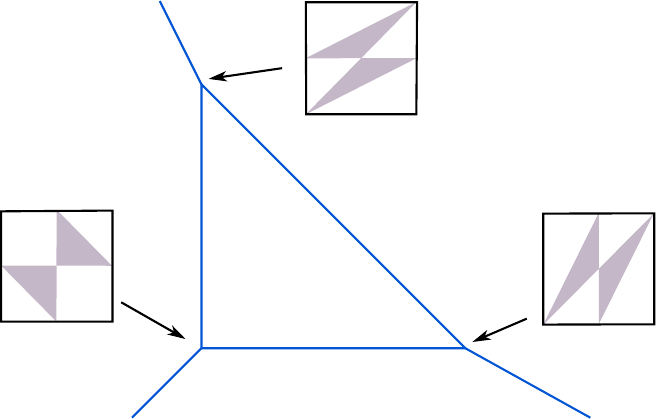}
\caption{} \label{pl_lift}
\end{center}
\end{figure}
\end{ex}
 
Given a $k$-dimensional polyhedron $\check e$ of $\Xi$, define the star-neighborhood of $\check e$ to be the union of the polyhedra of $\Xi$ which contain $\check e$, i.e.
\begin{equation} \label{star_neigh}
    \Xi_{\check e} = \bigcup_{f \preceq e, \ \dim f \geq 1} \ \check f.
\end{equation}
Similarly define its lift
 \[ \hat \Xi_{\check e} = \bigcup_{f \preceq e \ \dim f \geq 1 } \hat f. \]

\section{Lagrangian pairs of pants} \label{LagPants}
In this section we consider the torus $T = N_{\R}/N$ and its cotangent bundle $T^*T = M_{\R} \times T$. We assume that coordinates $y=(y_1, \ldots, y_{n+1})$ are chosen on $T$ so that $T$ is as in \eqref{toro}. The dual coordinates (see \S \ref{setup}) on the cotangent fibre $M_{\R}$ are $x=(x_1, \ldots, x_{n+1})$ 

\subsection{The construction} \label{theconstr}
We construct a smooth Lagrangian lift of the standard tropical hyperplane which we call a Lagrangian pair of pants. It can be considered as a Lagrangian smoothing of the PL-pair of pants $\hat \Gamma$ 


 \begin{defi} \label{blup} The {\it real blow up of the coamoeba} $C$ at its vertices is the smooth manifold $\tilde C$ defined as follows. If $p$ is one of the vertices $p_0, \ldots, p_{n+1}$ of $C$ and $U_p \subset C$ a small neighborhood of $p$ let 
 \[ \tilde U_p = \{ (q, \ell) \in U_p \times \RP^{n} \, | \, q-p \in \ell \}, \]
 where $\ell$ is a line through the origin in $\R^{n+1}$, which we think as a point of $\RP^{n+1}$. The real blow up of $C$ at $p$ is formed by gluing $\tilde U_p$ to $C-p$ via the projection map $\tilde U_p \rightarrow U_p$. We define $\tilde C$ to be the blow up of $C$ at all vertices. We denote by $\pi: \tilde C \rightarrow C$ the natural projection. Clearly we can identify $\tilde C - \cup_{j=0}^{n+1} \pi^{-1}(p_j)$ with $C - \{p_0, \ldots, p_{n+1} \}$ via $\pi$.   Notice that when $E_J$ is a face of $C$ of dimension greater that $1$, then it is also a coamoeba inside a smaller dimensional torus, therefore we can define its blow-up which we denote by $\tilde E_J$. 
 \end{defi}

It is easy to see that when $n=1$, $\tilde C$ is diffeomorphic to $S^2$ with three punctures. Let $G$ be the group acting on $C$ defined in \S \ref{symmetries}. We have the following easy fact

\begin{lem} \label{symmblup} The action of $G$ on $C$ lifts to an action on $\tilde C$.
\end{lem}

The goal is to construct a Lagrangian embedding $\Phi: \tilde C \rightarrow  T^*T$.   We say that an embedding $\Phi$  is a graph over $\tilde C$ if there exists a smooth map $\mathbf{h}: \tilde C \rightarrow M_{\R}$ such that $\Phi$ is of the type
\begin{equation} \label{phi:graph}
  \Phi(q) = (\mathbf{h}(q), \pi(q) ).
\end{equation}

Clearly, $\Phi(\tilde C - \cup_{j=0}^{n+1} \pi^{-1}(p_j))$ is the graph over $C - \{p_0, \ldots, p_{n+1} \}$ of the map $\h$. 
Given a smooth function $F: C - \{p_0, \ldots, p_{n+1} \} \rightarrow \R$, we can construct the graph of the exact one form $dF$, i.e. the graph of the map $\h$ defined by
\[ \h =( F_{y_1}, \ldots, F_{y_{n+1}}), \]
where $F_{y_j}$ denotes the partial derivative of $F$ with respect to $y_j$. In this case the graph is Lagrangian, but in general the map $\h$ does not extend smoothly to $\tilde C$.  We say that $\h$ defines the {\it graph of an exact one form over $\tilde C$} if both $F$ and $\h$ extend smoothly to $\tilde C$ and the map $\Phi$ defined by $\h$ via \eqref{phi:graph} is an embedding. By continuity, $\Phi(\tilde C)$ continues to be Lagrangian. The following is a key example.

\begin{ex} \label{exFloc} Let us study an example of a graph of an exact one form defined on the blowup $\tilde U_p$ of a neighborhood $U_p$ of a vertex $p$. Assume for convenience that $p$ is the point $p_0 = 0$. Let $U_p^{\pm} = U_p \cap C^{\pm}$. Define 
\begin{equation} \label{Floc} 
                              F(y) = \begin{cases} 
                                                \left( \prod_{j=1}^{n+1} y_j \right)^{\frac{1}{n+1}} \quad \text{on} \ U_p^+, \\
                                                \ \\
                                                (-1)^{n} \left( \prod_{j=1}^{n+1} y_j \right)^{\frac{1}{n+1}} \quad \text{on} \ U_p^-.
                                          \end{cases} 
                                                                      \end{equation}
Notice that $F$ is always well defined and odd. Let us prove that $\h=( F_{y_1}, \ldots, F_{y_{n+1}})$ extends to a smooth map on $\tilde U_p$. 
We have 
\[ F_{y_j} =  \frac{F(y)}{(n+1)y_j }.  \]                                                    
Coordinates on $\tilde U_p$ are given by $(\alpha_1, \ldots, \alpha_{n}, t)$, with $t \in \R$ and $\alpha_j >0$, and the map $\pi$ is
\begin{equation} \label{coords_blwup}
        \pi: (\alpha_1, \ldots, \alpha_{n}, t) \mapsto (t\alpha_1, \ldots, t\alpha_{n}, t) 
\end{equation}
In these coordinates, $\h$ lifted to $\tilde U_p$ becomes
\[ \h(t, \alpha_2, \ldots, \alpha_{n+1}) =\left( \frac{\left( \prod \alpha_j \right)^{\frac{1}{n+1}}}{(n+1)\alpha_1}, \, \ldots,  \, \frac{\left( \prod \alpha_j \right)^{\frac{1}{n+1}}}{(n+1)\alpha_{n}}, \, \frac{\left( \prod  \alpha_j \right)^{\frac{1}{n+1}}}{n+1} \right), \]      
which clearly is a smooth function on $\tilde U_p$. Similarly also $F$ extends to a smooth function on $\tilde U_p$. Notice that we have the identity
\[ (n+1)^{n+1} \prod_{k=1}^{n+1} F_{y_k} = 1 \]
thus the image of $\h$ is contained in the hypersurface
\begin{equation} \label{bound_amoeba}
 \mathcal S_0: \quad  (n+1)^{n+1} x_1 \ldots x_{n+1} = 1 \ \text{and} \  x_j > 0, \forall j.
\end{equation}
Indeed $\h$ restricted to $t=0$ gives a diffeomorphism from $\pi^{-1}(p_0)$ to this hypersurface. To see this, identify $\pi^{-1}(p_0)$ with $(\R_{>0})^{n}$ via the coordinates $(\alpha_1, \ldots, \alpha_{n})$. Then an inverse of $\h$ restricted to  $\pi^{-1}(p_0)$ is the map
\[ (x_1, \ldots, x_{n+1}) \mapsto \left( \frac{x_{n+1}}{x_1}, \ldots, \frac{x_{n+1}}{x_n} \right) \]
restricted to $\mathcal S_0$. This implies that $\Phi$ is an embedding of $\tilde U_p$.
\end{ex}

We can now give our global example. Define the following function $F$ on  $C$: 
\begin{equation} \label{Fglob} 
   F(y)=  \begin{cases} 
                                                \left( \cos \left( \sum_{j=1}^{n+1}y_j \right) \prod_{j=1}^{n+1} \sin y_j  \right)^{\frac{1}{n+1}} \quad \text{on} \ C^+, \\
                                                \ \\
                                               (-1)^{n} \left( \cos \left( \sum_{j=1}^{n+1}y_j \right) \prod_{j=1}^{n+1} \sin y_j  \right)^{\frac{1}{n+1}} \quad \text{on} \ C^-.
                                          \end{cases}
\end{equation}
We have that $F$ is well defined on $C$ and vanishes on the boundary of $C$. Moreover $F$ has a similar structure as the function in \eqref{Floc} and in fact it is modeled on it.  Notice that as $y$ approaches $p_0$, the factors $\sin y_j$ are asymptotic to $y_j$, moreover the first factor appearing in the expression of $F$ does not vanish at $p_0$. Thus $F$ vanishes at $p_0$ to the same order as the function defined in \eqref{Floc}.      The reason for the first factor in $F$ comes from requiring that $F$ has the same order of vanishing at all other vertices as well.   Let $G$ be the group acting on $C$ defined in \S \ref{symmetries}, then we have
\begin{lem} \label{Ginv} The function $F$ is $G$ invariant.
\end{lem}

 These facts guarantee the following. 

\begin{lem} \label{smoothExt} Let $F: C \rightarrow \R$ be defined as in \eqref{Fglob}. Then $\h= (F_{y_1}, \ldots, F_{y_{n+1}})$ extends smoothly to a map on $\tilde C$ and the map $\Phi$ defined in \eqref{phi:graph} is a Lagrangian embedding of $\tilde C$. 
\end{lem}

\begin{proof} Let us prove that $\h$ extends smoothly to the blow up $\tilde U_p$ of a neighborhood $U_p$ of the vertex $p_0$. Up to a sign, which we ignore for the moment, we have that 
\begin{equation} \label{fy}
     F_{y_j} = \frac{\cos \left( 2y_j + \sum_{k \neq j} y_k \right) \prod_{k \neq j} \sin y_k}{(n+1)\left( \cos \left(\sum_{k=1}^{n+1}y_k \right) \prod_{k=1}^{n+1} \sin y_k \right)^{\frac{n}{n+1}}} 
\end{equation}
The first factors in the numerator and denominator do not vanish at $p_0$ so we can ignore them. We treat the remaining factors by passing to the coordinates $(\alpha_1, \ldots, \alpha_{n}, t)$ on $\tilde U_p$ as in Example \ref{exFloc}. We have that if $j\neq n+1$, then
\[ \begin{split} 
  \frac{ \prod_{k \neq j} \sin y_k }{\left( \prod_{k=1}^{n+1} \sin y_k \right)^{\frac{n}{n+1}}} & =  \frac{ \sin t \,  \prod_{k \neq j} \sin (t  \alpha_k) }{\left( \sin t \, \prod  \sin (t \alpha_k) \right)^{\frac{n}{n+1}}} = \\
  \ & \ \\
  \ &= \pm \,  \frac{  \frac{\sin t}{t} \, \prod_{k \neq j} \frac{\sin (t \alpha_k)}{t} }{\left( \frac{\sin t}{t} \, \prod \frac{\sin (t \alpha_k)}{t}\right)^{\frac{n}{n+1}}},
  \end{split} \]
where the sign in the last equality is determined by the expression $\frac{t^n}{(t^{n+1})^{\frac{n}{n+1}}}$, which is $(-1)^n$ when $t<0$ and $1$ when $t>0$. Of course the functions $\frac{\sin (t\alpha_k)}{t}$ extend smoothly across $t=0$ to a non zero value. Thus $F_{y_j}$ extends smoothly to $\tilde U_p$.  The factor $(-1)^n$ in \eqref{Fglob} guarantees that the expressions for $F_{y_j}$ on $U^{+}_p$ and $U^-_{p}$ extend smoothly to the same function (as in Example \ref{exFloc}). Similarly we have that when $j=n+1$
\[
  \frac{ \prod_{k \neq n+1} \sin y_k }{\left( \prod_{k=1}^{n+1} \sin y_k \right)^{\frac{n}{n+1}}} = \pm \,  \frac{ \prod \frac{\sin (t \alpha_k)}{t} }{\left( \frac{\sin t}{t} \, \prod \frac{\sin (t \alpha_k)}{t}\right)^{\frac{n}{n+1}}},
  \]
which is smooth.
 We also have 
\begin{equation} \label{ht0}
\h|_{t=0} =\left(  \, \frac{\left( \prod \alpha_j \right)^{\frac{1}{n+1}}}{(n+1)\alpha_1}, \, \ldots,  \, \frac{\left( \prod \alpha_j \right)^{\frac{1}{n+1}}}{(n+1)\alpha_{n}}, \,  \frac{\left( \prod  \alpha_j \right)^{\frac{1}{n+1}}}{n+1}\right),
\end{equation}
which is the same function as in Example \ref{exFloc}. Thus $\h$ maps the set $\pi^{-1}(p)$ diffeomorphically onto the hypersurface $\mathcal S_0$ defined in \eqref{bound_amoeba}. 

By the symmetries of $F$ given in Lemma \ref{Ginv} the behavior of $F$ at the other vertices $p_k$ is the same as in $p_0$. This completes the proof of the Lemma.
\end{proof}

\begin{rem} \label{Fsmooth}
Observe that the function \eqref{Fglob} is smooth when lifted to $\tilde C$. In fact, if we write $F$ in the above coordinates on $\tilde U_{p_o}$
\begin{equation*}
 F=  t \left( \cos \left( t \left( 1+ \sum \alpha_j \right) \right) \frac{\sin t}{t}\prod \frac{\sin (t  \alpha_j)}{t} \right)^{\frac{1}{n+1}}.                                       
\end{equation*}
which is smooth.
\end{rem}
 
\begin{defi} \label{LpPants}  We call the submanifold $L=\Phi(\tilde C)$ the {\it standard $(n+1)$-di\-men\-sio\-nal Lagrangian pair of pants}. Given $\lambda >0$, let $\Phi_{\lambda}$ be the embedding constructed from $\h_{\lambda} = (\lambda F_{y_1}, \ldots, \lambda F_{y_{n+1}})$ via \eqref{phi:graph}. Then,  if $\lambda \neq 1$, we call $\Phi_{\lambda}(\tilde C)$ a {\it rescaled Lagrangian pair of pants}
\end{defi}

We have the following consequence of Lemmas \ref{symmblup} and \ref{Ginv}
\begin{lem} \label{eq_action} Given a transformation $R_k$ as in \S \ref{symmetries}, the map $\h: \tilde C \rightarrow M_{\R}$ defined using $F$ satisfies
\[ \h(R_k(y)) = R^*_k \h(y). \]
In particular the group $\mathcal G$ acts on an $(n+1)$-dimensional Lagrangian pair of pants.
\end{lem}

We have one last symmetry. Consider the involution of the torus $\iota: [y] \mapsto [-y]$. Clearly $\iota$ maps $C$ to itself and exchanges $C^+$ with $C^-$. 

\begin{lem} \label{invo}The function $F$ satisfies $F(\iota(y)) = -F(y)$ and $\h$ satisfies $\h(\iota(y)) = \h(y)$. Therefore $\iota$ acts on a Lagrangian pair of pants. 
\end{lem}
\subsection{The image of $\h$} \label{Imh} Define the following subsets:
\[ \begin{split} 
           \mathcal H_{0} &= \left \{ (x_1, \ldots, x_{n+1}) \, | \, x_j \geq 0 \ \text{and} \ x_1 x_2 \ldots x_{n+1} \leq  \frac{1}{(n+1)^{n+1}} \right \}, \\
           \mathcal H_{k} &= R^*_{k} \mathcal H_{0}.
   \end{split}        
           \]
Recall that we defined $J_{k}$ to be the complement of $k$ in $\{0, \ldots, n+1 \}$ (see \S \ref{stCoam}). Then $\mathcal H_0 \subset \Gamma_{J_0}$ and $\mathcal H_{k} \subset \Gamma_{J_k}$, see \eqref{action_cones}. An extended description of $\mathcal H_{k}$ is
\[ \mathcal H_{k} = \left \{ t_k u_0 + \sum_{l \neq k} t_l u_l \, | \, t_l \geq 0 \ \forall l \ \text{and} \ t_1t_2 \ldots t_{n+1} \leq  \frac{1}{(n+1)^{n+1}} \right\}. \]
Let 
\begin{equation} \label{the amoeba}
        \mathcal H = \bigcup_{l=0}^{n+1} \mathcal H_l.
\end{equation}
If $\mathcal S_0$ is the hypersurface defined in \eqref{bound_amoeba} let 
\begin{equation} \label{bound_amoeba_k}
   \mathcal S_k = R^*_{k} \mathcal S_{0}.
\end{equation}
Then the boundary of $\mathcal H$ is 
\[ \partial \mathcal H = \bigcup_{l=0}^{n+1} \mathcal S_l. \]
In this section we will prove the following
\begin{prop} \label{imh} Assume $n=1$ or $2$. The image of $\h: \tilde C \rightarrow M_{\R}$ is $\mathcal H$. Moreover $\h$ defines a diffeomorphism between $\inter C^+$ and $\inter \mathcal H$. 
\end{prop}

\begin{figure}[!ht] 
\begin{center}
\includegraphics{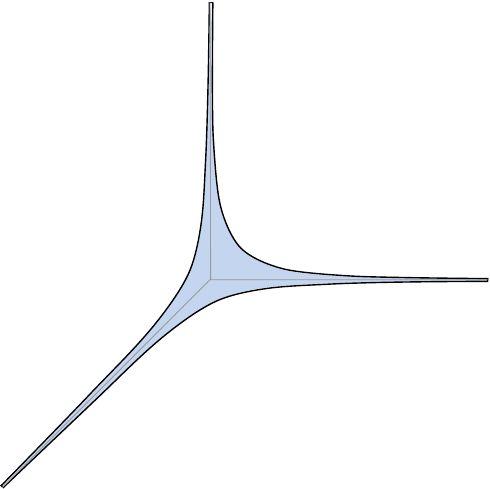}
\caption{} \label{lagrangian_amoeba}
\end{center}
\end{figure}

The statement must be true for all values of $n$, but unfortunately we have a complete proof only in these dimensions, the main difficulty is proving that $\h$ is a local diffeomorphism. We prove this for $n=1$ and $2$ in the Appendix, Proposition \ref{imh23}.

Figure \ref{lagrangian_amoeba} depicts $\mathcal H$ in the case $n=1$. It clearly resembles the amoeba of a complex hyperplane in $(\C^*)^2$.  Lemma \ref{invo} allows us to restrict to $C^+$. 


\begin{defi} For every pair of vertices $p_k$ and $p_j$ of $C^+$, let $\delta_{jk}$ be the hyperplane that contains all vertices different from $p_k$ and $p_j$ and passes through the middle point of the edge from $p_k$ to $p_j$. This hyperplane cuts $C^+$ in two halves. We denote by $\Delta_{jk}$ the half which contains $p_k$.
\end{defi}

Clearly, the set of hyperplanes $\delta_{jk}$ cuts $C^+$ into the first barycentric subdivision of $C^+$. 
We have the following inequalities defining $\Delta_{jk}$
\begin{equation} \label{inqDelta1} 
     \Delta_{j0} =  \left \{ y \in C^+ \, | \, 2y_j + \sum_{k\neq j} y_k \leq \frac{\pi}{2} \right \} 
\end{equation}
and when $j,k \neq 0$
\begin{equation} \label{inqDelta2} 
 \Delta_{jk} = \left \{ y \in C^+ \, | \,  y_k-y_j \geq 0 \right \}.
\end{equation}
For every face $E_J^+$ of $C^+$ let $\mathcal W_{J}^+$ denote its star neighborhood, i.e. the union of simplices of the barycentric subdivision whose closures contain the barycenter of $E_J^+$. We have that 
\begin{equation} \label{facenbh}
       \mathcal W_{J}^+ = \bigcap_{k \notin J, j \in J} \Delta_{jk}.
\end{equation}
As usual we denote by $\mathcal W_{J}^-$ the image of $\mathcal W_{J}^+$ with respect to $\iota$ and 
\[ \mathcal W_{J} = \mathcal W_{J}^- \cup \mathcal W_{J}^+,\]
\[ \tilde{\mathcal W}_{J}= \pi^{-1}(\mathcal W_{J}). \]

We have a dual structure for $\mathcal H$. 

\begin{defi}
For every $j,k =0, \ldots, n+1$ with $j \neq k$ let 
\[ d_{jk} =  \spn_{\R} \{ u_l \, | \, l \neq j,k \} \]
It is a codimension $1$ vector subspace which divides $M_{\R}$ in two halves. Denote by $D_{jk}$ the half which contains $u_j$. 
\end{defi}

We have the following inequalities defining $D_{jk}$
\[ D_{j0}=  \left \{ x \in M_{\R} \, | \, x_j \geq 0 \right \} \]
and when $j,k \neq 0$
\[ D_{jk} = \left \{ x \in M_{\R}  \, | \,  x_j-x_k \geq 0 \right \}. \]
Let 
\begin{equation} \label{h_nbhoods}
  \mathcal V_{J} = \bigcap_{j \in J, k \notin J} D_{jk}
\end{equation}
When $1 \leq |J| \leq n$, $\mathcal V_{J}$ contains the face $\Gamma_J$ of $\Gamma$ and can be regarded as a neighborhood of it, analogous to the star neighborhood $\mathcal W_{J}$ of the face $E_{J}$. Moreover 
\[ \mathcal V_{J_k} \cap \mathcal H = \mathcal H_k.\]

We have the following useful facts:

\begin{lem} \label{tnbhd}
\[ \begin{split}
             R^*_l( \mathcal V_{J}) &= \mathcal V_{R_lJ}, \\
             R_l(\mathcal W_{J}) &= \mathcal W_{R_lJ}. 
   \end{split}  \]
\end{lem}

\begin{proof} These are easy consequences of the definitions \end{proof}

Combined with Lemma \ref{eq_action}, the above lemma tells us that the behavior of $\h$ in a neighborhood of any face is determined by its behavior in a specific one.
Recall that a vertex $p_k$ is the face $E_{J_k}$ where $J_k$ is the complement of $k$. 
\begin{lem} \label{hvertexnbd} \[ \h(\tilde{\mathcal W}_{J_k}) \subseteq \mathcal H_{k}. \]
\end{lem}
\begin{proof} By Lemmas \ref{eq_action}, \ref{tnbhd}, \ref{invo} and by the definition of $\mathcal H_k$, it is enough to prove that 
\[ \h(\mathcal W_{J_0}^+)\subseteq \mathcal H_{0}. \]
We have that $h(y) \in \mathcal H_{0}$ if and only if $F_{y_j}(y) \geq 0$ for all $j$ and 
\[ \prod_{j=1}^{n+1} F_{y_j}(y) \leq \frac{1}{(n+1)^{n+1}}. \]
On the other hand, $y \in \mathcal W_{J_0}^+$ satisfies
\begin{equation} \label{wpk} 
            0 \leq y_j \leq  \frac{\pi}{2} \quad \text{and} \quad 0 \leq \sum_{k=1}^{n+1} y_k \leq2y_j + \sum_{k \neq j} y_k \leq \frac{\pi}{2}, \quad  \forall j.
 \end{equation}
This follows from  \eqref{facenbh} and \eqref{inqDelta1}. In particular $F_{y_j} \geq 0$ on  $\mathcal W_{J_0}^+$ since all factors in \eqref{fy} are non-negative. Moreover
\[ \prod_{j=1}^{n+1} F_{y_j} = \frac{\prod_{j=1}^{n+1} \cos \left(2y_j + \sum_{k \neq j} y_k \right)}{(n+1)^{n+1} \left(\cos \left(\sum_{k=1}^{n+1}y_k \right) \right)^n}. \]
Inequalities \eqref{wpk} imply that  
\[ \prod_{j=1}^{n+1} F_{y_j} \leq \frac{1}{(n+1)^{n+1}} \cos \left( \sum_{k=1}^{n+1}y_k \right) \leq \frac{1}{(n+1)^{n+1}}. \]
\end{proof}

\begin{figure}[!ht] 
\begin{center}
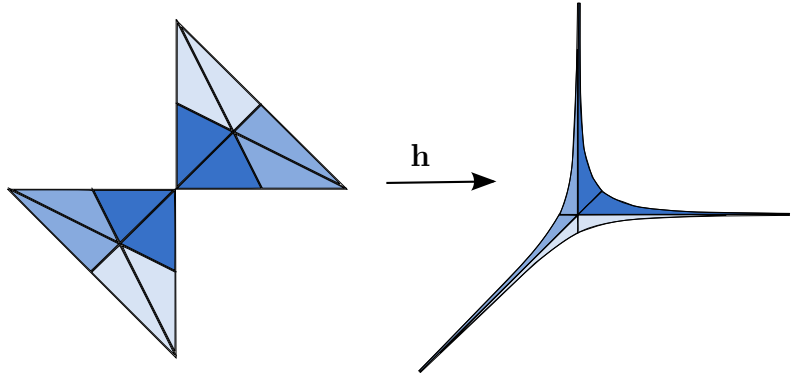
\caption{The images of star neighborhoods of vertices. Areas with same shading (color) are matched by $\h$.} \label{amoeba_bar_subdiv1}
\end{center}
\end{figure}

\begin{cor} \label{immagineH}  \[ \h(\tilde C) \subseteq \mathcal H. \]
\end{cor} 

The following lemma describes the behavior of $\h$ near the boundary of $C^+$.
\begin{lem} \label{hnear_faces} 
Let $E_J$ be a face of $C$ of codimension $1 \leq |J| \leq n$ and let $\{ q_\ell \}$ be a sequence of points of $C$ which converges to a point in $\inter{E_J}$. Then we have the following behaviour of $\h$. If $p_0$ is a vertex of $E_J$ (i.e. $0 \notin J$) then
\[ \begin{split}
          \lim h_{j}(q_\ell) &= + \infty \quad \forall j \in J,  \\
          \lim h_{k}(q_{\ell}) & = 0 \quad \forall k \notin J \cup \{ 0 \}.
   \end{split}
\]
If $p_0$ is not one of the vertices of $E_J$ (i.e. $0 \in J$), then for all $i \notin J$ we have 
\[ \begin{split}
          \lim h_{i}(q_\ell) &= - \infty,  \\
          \lim h_{j}(q_{\ell})- h_{i}(q_{\ell}) & = + \infty \quad \forall j \in J - \{ 0 \}, \\
          \lim h_{k}(q_{\ell})- h_{i}(q_{\ell}) & = 0 \quad \forall k \notin J. 
   \end{split}
\]
\end{lem}

\begin{proof} The case when $p_0$ is a vertex of $E_J$ follows directly from the explicit formula \eqref{fy} for $h_j$. If $p_0$ is not a vertex of $E_J$ then for all $i \notin J$ we can apply a transformation $\mathcal R_i$ defined in \eqref{glob_symm} to reduce to the first case. The resulting behavior is the one described. 
\end{proof}

\begin{cor} \label{h_near_bndry} If $\{ q_k \}$ is a sequence of points of $C$ which converges to a point on the boundary of $C$, then either $\{ \h(q_k) \}$ converges to a point on the boundary of $\mathcal H$ or $\lim_{k \rightarrow +\infty} || h(q_k)|| = + \infty$. 
\end{cor}
\begin{proof} This follows from the previous Lemma if $q_k$ converges to the interior of a proper face $E_J$ of codimension at most $n$ and from a direct inspection when $q_k$ converges to a vertex. 
\end{proof}

\begin{proof}[Proof of Proposition \ref{imh}] In the Appendix, Proposition \ref{imh23}, we prove that $\h$, restricted to $\inter C^+$ is a local diffeomorphism.  Moreover Corollary \ref{HessFneg} tells us that the Hessian of $F$ is negative definite. The gradient of a function with negative definite Hessian defined on a convex set must be an injective map. Indeed let $\gamma(t)$ be a segment joining two distinct points $q_1, q_2 \in \inter C^+$ with direction $v = q_1 - q_2$, then 
\begin{equation} \label{injectiv_ineq} 
\inn{\h(q_2)-\h(q_1)}{v} = \int_0^1 \inn{\frac{d \h(\gamma(t))}{dt}}{v} dt = \int_0^1 \inn{\hes{F}(v)}{v} dt < 0.
\end{equation}
Therefore $\h$, restricted to $\inter C^+$, is injective. Moreover Corollary \ref{immagineH} implies that $\h(\inter C^+) \subset \inter \mathcal H$. We have already seen (see Example \ref{exFloc} and Lemma \ref{smoothExt}) that $\h$ restricted to $\pi^{-1}(p_0)$ is a diffeomorphism onto the hypersurface $\mathcal S_0$, which forms one component of the boundary of $\mathcal H$. Similarly the images of the sets $\pi^{-1}(p_k)$ are the hypersurfaces $\mathcal S_k$ defined in \eqref{bound_amoeba_k}, which give the other components. This, together with Corollary \ref{h_near_bndry}, implies that $\h(\tilde C) = \mathcal H$ and $\h(\inter C^+) = \inter \mathcal H$ for topological reasons.
\end{proof}

All of the above also implies

\begin{lem} \label{hNbhd}
\[ \h( \tilde{\mathcal W}_{J}) = \mathcal V_{J} \cap \mathcal H \]
\end{lem}

\begin{proof} It follows from the fact that 
\[ \h( \Delta_{jk}) \subseteq D_{jk} \]
Using Lemmas \ref{eq_action} and \ref{tnbhd} it is enough to prove the latter inclusion for the cases $\Delta_{j0}$. This follows from \eqref{inqDelta1}, which implies that if $y \in \Delta_{j0}$, then $F_{y_j}(y) \geq 0$, i.e.  $\h(y) \in D_{j0}$. The equality of the two sets follows from surjectivity of $\h$.
\end{proof}

The following Corollary gives Theorem \ref{main_thm} for the tropical hyperplane.

\begin{cor} \label{hausdorf_conv} Let $L_{\lambda} = \Phi_{\lambda}(\tilde C) \subset T^*T$ be the family of rescaled Lagrangian pairs of pants (see Definition \ref{LpPants}) for $\lambda \in (0,1)$. Then, as $\lambda \to 0$, $L_{\lambda}$ converges in the Hausdorff topology to the PL-lift $\hat \Gamma$ of the tropical hyperplane $\Gamma$.
\end{cor}

\begin{proof} It is clear that the projection of $L_{\lambda}$ to $M_{\R}$, which is equal to the set $\lambda \mathcal H$, converges to $\Gamma$.  The convergence of $L_{\lambda}$ to $\hat \Gamma$ follows easily from the properties $\h$ proved above.
\end{proof}
\subsection{Projections to faces and Legendre transform} \label{projections} The fact proved in Proposition \ref{imh} that $\h$ gives a diffeomorphism between $\inter C^+$ and $\inter \mathcal H$ implies that $\Phi(\inter C^+)$, i.e. half of the Lagrangian pair of pants, is also the graph of an exact one form defined over $\inter \mathcal H$. This is the classical Legendre transform. Indeed let 
\[ G(x) = \inn{x}{y} - F(y), \]
where $\inn{.}{.}$ is the standard duality pairing between $x$ and $y$. Then an elementary calculation shows that  
\[ \frac{\partial G}{\partial x_j} = y_j, \]
which proves that $\Phi(\inter C^+)$ is the graph of $dG$. We will now generalize this to show that we can break up $\Phi(\tilde C)$ into parts which are graphs of exact one forms over the ends $\hat{\Gamma}_{J}$ (see \eqref{ends_lift}) of the $PL$-lift $\hat{\Gamma}$. 

\begin{defi} \label{projFace1} Given a face $E_J$ of $C$ of codimension $1 \leq|J| \leq n$, let $L \subseteq N_{\R}$ be a vector subspace of dimension $|J|$ which is transversal to $E_J$. Let $U_{J, L}$ be the set of points $y \in \inter C$ such that there exists a $y' \in \inter E_J$ such that $y-y' \in L$. If such a $y'$ exists, it is unique by transversality. Thus we can define the projection 
\[ 
\begin{split}
 \y_{J, L}: U_{J, L} & \rightarrow \inter E_{J} \\
                        y & \mapsto y'.
 \end{split}
 \]
Recall that $\{ p_k \}_{k \notin J}$ is the set of vertices of $E_J$. Define 
\[ \tilde U_{J,L} = \pi^{-1}(U_{J,L} \cup \{ p_k \}_{k \notin J})  \subseteq \tilde C\]
Then $\y_{J,L}$ extends to a map $\y_{J,L}: \tilde U_{J,L} \rightarrow \tilde E_{J}$. 
\end{defi}

Dually we give the following definition.  

\begin{defi} \label{projFace2} Let $\Gamma_J$ be a face of $\Gamma$. Recall that we denoted by $V_J$ the smallest subspace containing $\Gamma_J$. Let $L$ be as in Definition \ref{projFace1}. Define 
\[ L^{\perp} = \{ x \in M_{\R} \, | \, \inn{x}{y} = 0 \ \forall y \in L \} \]
Then $L^{\perp}$ has dimension $n+1-|J|$ and it is transversal to $V_J$. It thus defines the projection   $\x_{J,L}: M_{\R}  \rightarrow V_J$, dual to $\y_{J,L}$, whose fibres are parallel to $L^{\perp}$.  
\end{defi}

Given a face $E_J$ of $C$, let $T_J$ be the smallest subtorus of $T$ which contains $E_J$. By construction $V_J \times T_J$ is a Lagrangian submanifold of $M_{\R} \times T$.  Given $L$ and $L^{\perp}$ as in Definitions  \ref{projFace1} and  \ref{projFace2}, the space $(V_J \times T_J) \times (L^{\perp} \times L)$ is naturally a covering of $M_{\R} \times T$ and thus induces from the latter a symplectic form. We have the following

\begin{lem} \label{cotangent} The choice of a vector subspace $L$ as in Definition \ref{projFace1} induces a natural (linear) symplectomorphism between the cotangent bundle of $V_J \times T_J$ and $(V_J \times T_J) \times (L^{\perp} \times L)$.
\end{lem}
 
\begin{proof} This is just linear algebra. In fact $L^{\perp} \times L$ can be naturally identified with a cotangent fibre of $V_J \times T_J$ by sending the pair $(\ell', \ell) \in L^{\perp} \times L$ to the linear form $(v, w) \mapsto \inn{\ell}{v} - \inn{\ell'}{w}$, where $v$ is a tangent vector in $V_J$ and $w$ in $T_J$. The signs in this identification are chosen in order to match the symplectic forms. \todo{check this!}
\end{proof}
 Given $L$, $U_{J,L}$ and $\tilde U_{J,L}$ as above,  define   $\h_{J,L}: \tilde U_{J,L} \rightarrow V_{J}$ to be the map  
\begin{equation*}  
                  \h_{J,L} = \x_{J,L} \circ \h  
\end{equation*}
and $\gb_{J,L}: \tilde U_{J,L} \rightarrow  V_{J} \times \tilde{E}_J$ to be 
\begin{equation} \label{mapg}
                  \gb_{J,L} = ( \y_{J,L}, \, \h_{J,L} ).
\end{equation}

\begin{ex} \label{proj_st1} Given $k \notin J$, consider the $|J|$-dimensional face of $C$ whose vertices are $p_k$ and all $p_j$'s which are not vertices of $E_J$ (i.e. $j \in J$). Let $L$ be the $|J|$ dimensional subspace parallel to this face. In this case $U_{J,L} = \inter C$. We denote the corresponding projection by $\y_{J,k}$.  We denote the projection dual to $\y_{J,k}$ in the sense of Definition \ref{projFace2} by $\x_{J,k}$. For instance, when $J=\{1, \ldots, \ell \}$ and $k=0$ we have
\[ \y_{J,0}(y) = (y_{\ell+1}, \ldots, y_{n+1}) \quad \text{and} \quad \x_{J,0}(x)= (x_1, \ldots, x_{\ell}) \]
and the maps $\h_{J,0}$ and $\gb_{J,0}$ become
\begin{equation} \label{g_special_case} 
        \h_{J,0} = (h_1, \ldots, h_{\ell} ) \quad \text{and} \quad \gb_{J,0} =(y_{\ell+1}, \ldots, y_{n+1}, h_1, \ldots, h_{\ell}). 
\end{equation}
\end{ex}

We have the following 
\begin{prop} \label{faceproj} Assume $n=1$ or $2$. The map $\gb_{J,L}: \tilde U_{J,L} \rightarrow \tilde{E_J} \times V_{J}$ is a diffeomorphism onto the open subset $Z_{J,L} = \gb_{J,L}(\tilde U_{J,L}) \subseteq  \tilde{E_J} \times V_{J}$. Moreover, via the identification of the cotangent bundle of $V_J \times T_J$ with (a covering of) $M_{\R} \times T$ given in Lemma \ref{cotangent}, $\Phi( \tilde U_{J,L})$ is the graph of an exact one form over  $Z_{J,L}$ obtained as the differential of a Legendre transform of $F$.
\end{prop}

\begin{proof} Using the symmetries we can assume that $J=\{1, \ldots, \ell \}$ for some $\ell \leq n$. In this case the torus $T_J$ is spanned by the vectors $\{ u^*_{\ell+1}, \ldots, u^*_{n+1} \}$. On the dual, $V_J$ is spanned by the vectors $\{ u_1, \ldots, u_{\ell} \}$. We can choose a basis $\{ u^{*}_1, \ldots, u^*_\ell \}$ of $L$ and a basis $\{ u_{\ell+1}, \ldots, u_{n+1} \}$ of $L^{\perp}$ so that $\{u^*_{1}, \ldots, u^*_{n+1} \}$ and $\{u_1, \ldots, u_{n+1} \}$ are dual basis. Then, in the corresponding coordinates $(x,y)=(x_1, \ldots, x_{n+1}, y_1, \ldots, y_{n+1})$ the symplectic form on $M_{\R} \times T$ is the usual one \eqref{st_sympl} and the projections have the form
\[ \y_{J,L} = (y_{\ell+1}, \ldots, y_{n+1}) \quad \text{and} \quad \x_{J,L} =(x_1, \ldots, x_{\ell}). \]
With respect to these new coordinates $\h$ continues to be the gradient of $F$ and the Hessian of $F$ is still negative definite (Corollary \ref{HessFneg}). Let us first prove that $\gb_{J,L}$ is a local diffeomorphism. Away from the vertices of $E_J$ (i.e. in $U_{J,L}$) we have
\[
              \gb_{J,L}(y) =(y_{\ell+1}, \ldots, y_{n+1}, h_{1}(y), \ldots, h_{\ell}(y) ).
\]
The Hessian of $F$ restricted to a fibre of $\y_{J,L}$ is negative definite, hence
\[ \det ( F_{y_j y_k})_{1 \leq j,k \leq \ell } \neq 0. \]
Therefore $\gb_{J,L}$ is a local diffeomorphism. Let $p$ be a vertex of $E_J$. In the above coordinates, we can assume that $p=(0, \ldots, 0)$. 
Let $(\alpha_1, \ldots, \alpha_n, t)$ be the coordinates on a neighborhood $\tilde U_p$ of $\pi^{-1}(p)$ which satisfy \eqref{coords_blwup}. Then, in these coordinates,
\[
              \gb_{J,L} =(\alpha_{\ell+1}, \ldots, \alpha_{n}, t, h_{1}, \ldots, h_{\ell})
\]
We have that $\gb_{J,L}$ is a local diffeomorphism along $\pi^{-1}(p)$ if and only if the matrix
\begin{equation} \label{blwup_diffh}
 \left( \frac{\partial h_j}{\partial \alpha_k}\right)_{1 \leq j,k \leq \ell}
\end{equation}
is non-degenerate at $t=0$. We have that for all $j,k = 1, \ldots, n$
\[ \frac{\partial h_j}{\partial \alpha_k}|_{t=0} = \lim_{t \rightarrow 0} tF_{y_jy_k}(t\alpha_1, \ldots, t \alpha_n, t) \]
Therefore the matrix 
\[ \left( \frac{\partial h_j}{\partial \alpha_k}\right)_{1 \leq j,k \leq n} \]
is symmetric and negative semidefinite, since it is the limit of symmetric and negative definite matrices. We  know that $(h_1, \ldots, h_n)$ is a diffeomorphism when restricted to $\pi^{-1}(p)$ (see Lemma \ref{smoothExt}), therefore the above matrix must be negative definite. In particular also the matrix in \eqref{blwup_diffh} must be negative definite and hence $\gb_{J,L}$ is a local diffeomorphism along $\pi^{-1}(p)$.

We have that $\gb_{J,L}$ is injective if and only if $\h_{J,L}$ is injective when restricted to a fibre of $\y_{J,L}$. If $y' \in \inter E_J$ then $\y_{J,L}^{-1}(y')$ is convex and thus $\h_{J,L}$ restricted to it must be injective (see \eqref{injectiv_ineq}).  If $y' \in \pi^{-1}(p)$ for some vertex $p$ of $E_J$, then $y'=(\bar \alpha_{\ell+1}, \ldots, \bar \alpha_{n})$ for some fixed $\bar \alpha_j>0$ and 
\begin{equation} \label{preim_proj_blwp}
   \y_{J,L}^{-1}(y') = \{ (\alpha_1, \ldots, \alpha_{\ell}, \bar \alpha_{\ell+1}, \ldots, \bar \alpha_{n}) \, | \, \alpha_j \in \R_{>0}, j=1, \ldots, \ell \}.\end{equation} 
Therefore $\y_{J,L}^{-1}(y')$ is identified with a convex set and, as we have seen above, the differential of $\h_{J,L}$ restricted to this set is a negative definite symmetric matrix. Thus $\h_{J,L}$ is injective, by the same argument as in \eqref{injectiv_ineq}. This concludes the proof that $\gb_{J,L}$ is a diffeomorphism onto its image.

Let us prove the last claim of the proposition.  The identification of the cotangent bundle of $V_J \times T_J$ with $M_{\R} \times T$ given in Lemma \ref{cotangent} identifies the cotangent fibre coordinates with $(y_1, \dots, y_{\ell}, - x_{\ell+1}, \ldots, -x_{n+1})$. The fact that $\gb_{J,L}$ is a diffeomorphism onto $Z_{J,L}$ implies that $\Phi(U_{J,L})$ is the graph of an exact one form over $Z_{J,L}$. Consider the following Legendre transform of $F$:
\begin{equation} \label{legdre_transf}
  G(\y_{J,L}, \x_{J,L}) = -F(y) + \sum_{j=1}^{\ell} x_i y_i.
\end{equation}
Then a standard calculation gives
\[ \begin{split}
        \frac{\partial G}{\partial y_j} & = -h_j, \quad \forall j=\ell+1, \ldots, n+1 \\
         \frac{\partial G}{\partial x_k} & = y_k, \quad \forall k=1, \ldots, \ell \\
   \end{split}
                         \]
which implies that $\Phi( U_{J,L})$ is the graph of $dG$. It is clear that $G$ is well defined and smooth also when $\y_{J,L} \in \pi^{-1}(p)$ for some vertex $p \in E_J$, since all the functions involved in its definition are well defined and smooth. Moreover it is also clear that $dG$ smoothly extends to an embedding. 
\end{proof}

\begin{cor} \label{submers} The map $\h_{J,L}: \tilde U_{J,L} \rightarrow  V_{J}$ is a submersion. The fibres of $\h_{J,L}$ can be identified with open subsets of $\tilde E_{J}$ via the map $\y_{J,L}$. 
\end{cor}

In general $\h_{J,L}$ is not a fibre bundle, since it may happen that some fibres are connected, or compact, and some are not, but we now restrict to a case when $\h_{J,L}$ is a fibre bundle. First let us define certain neighborhoods of the faces $E_J$. Given $k \notin J$, let 
\[ \mathcal W^+_{J,k} = \bigcap_{j \in J} \Delta_{jk} \]
and as usual we define $\mathcal W^-_{J,k}$, $\mathcal W_{J,k}$ and its lift $\tilde{\mathcal{W}}_{J,k}$. Clearly we have that 
\[ \tilde{\mathcal W}_J = \bigcap_{k \notin J} \tilde{\mathcal W}_{J,k}. \]
Similarly define neighborhoods of $\Gamma_J$:
\[ \mathcal V_{J,k} =  \bigcap_{j \in J} D_{jk}.\]
We have 
\[ \mathcal V_{J} =  \bigcap_{k \notin J} \mathcal V_{J,k}.\]
Moreover, as in Lemma \ref{hNbhd}, we have that $\h(\tilde{\mathcal{W}}_{J,k})= \mathcal V_{J,k} \cap \mathcal H$.



\begin{cor} \label{fbr_bndl} Given a face $E_J$ of $C$ and $k \notin J$, consider the projections $\y_{J,k}$ and $\x_{J,k}$ as in Example \ref{proj_st1} and the associated maps $\h_{J,k}$ and $\gb_{J,k}$. Then 
\begin{equation} \label{g_fbr_bnd}
        \gb_{J,k}: \inter \tilde{\mathcal{W}}_{J,k} \rightarrow \tilde E_{J} \times \inter \Gamma_J
 \end{equation}
is a diffeomorphism. In particular $\Phi (\inter \tilde{\mathcal{W}}_{J,k})$ is the graph of the differential of a Legendre transform of $F$ defined on $\tilde E_{J} \times \inter \Gamma_J$. Moreover $\h_{J,k} : \inter \tilde{\mathcal{W}}_{J,k} \rightarrow  \inter \Gamma_{J}$ is a trivial fibre bundle over $\inter \Gamma_J$ with fibre $\tilde E_J$.
\end{cor}

\begin{proof} Using the symmetries we can assume that $J = \{ 1, \ldots, \ell \}$ and $k=0$. Then $\h_{J,0}$ and $\gb_{J,0}$ are as in \eqref{g_special_case} and $\inter \tilde{\mathcal{W}}_{J,0}$ is defined by the following inequalities (see \eqref{inqDelta1})
\[ 
    \inter \tilde{\mathcal{W}}_{J,0}: \quad  |2y_j + \sum_{k\neq j} y_k| < \frac{\pi}{2}, \quad j=1, \ldots, \ell. 
\]
In particular for all $j=1, \ldots, \ell$ we have that $h_j > 0$ on $\inter \tilde{\mathcal{W}}_{J,0}$. Therefore $\h_{J,0}(\inter \tilde{\mathcal{W}}_{J,0}) \subseteq \inter \Gamma_{J}$. We only need to show that $\gb_{J,0}$ as in \eqref{g_fbr_bnd} is surjective. This holds if and only if $\h_{J,0}$ restricted to $\inter \tilde{\mathcal{W}}_{J,0} \cap \y_{J,0}^{-1}(y')$ surjects onto $\inter \Gamma_J$ for all $y' \in \tilde E_J$. If $y' \in \inter E_J$, we can view $(y_1, \ldots, y_{\ell})$ as coordinates on $\y_{J,0}^{-1}(y')$. For $j=1, \ldots, \ell$, we have that $h_j(y) = 0$ on the points $y$ of the boundary of $\inter \tilde{\mathcal{W}}_{J,0} \cap \y_{J,0}^{-1}(y')$ which satisfy
\[ |2y_j + \sum_{k\neq j} y_k| = \frac{\pi}{2}. \]
On the other hand it follows from Lemma \ref{hnear_faces} that if $\{ q_{k} \}$ is a sequence of points of $\inter \tilde{\mathcal{W}}_{J,0} \cap \y_{J,0}^{-1}(y')$ which converges to the boundary of $C$ then
\[ \lim h_j(q_k) = + \infty, \]
for all $j=1, \ldots, \ell$. Therefore, by continuity, for any $x \in \inter \Gamma_J$, there must be a (unique) $y \in \inter \tilde{\mathcal{W}}_{J,0} \cap \y_{J,0}^{-1}(y')$ such that $\h_{J,0}(y) = x$. 

If $y' \in \pi^{-1}(p)$ for some vertex $p$ of $\tilde E_J$ then $\y_{J,0}^{-1}(y')$ is as in \eqref{preim_proj_blwp}. Here we have the explicit description of $\h$ (hence of $\h_{J,0}$) given in \eqref{ht0} and we can check directly that $\h_{J,0}$ surjects onto $\inter{\Gamma_J}$. \end{proof}

\subsection{Lagrangian pairs of pants are homeomorphic to the PL-lift of $\Gamma$}
Given a Lagrangian pair of pants $\Phi: \tilde C \rightarrow M_{\R} \times T$ we prove
\begin{prop}  \label{pairpants_PLft} The Lagrangian pair of pants $\Phi(\tilde C)$ is homeomorphic to the $PL$-lift $\hat \Gamma$ of $\Gamma$. 
\end{prop} 
\begin{proof} The case $n=1$ is obvious, since they are both homeomorphic to a sphere with three punctures. For simplicity, we only give a proof for the case $n=2$, where we use Proposition \ref{imh}.  In the definition of $\hat \Gamma$ given in \S \ref{plift} it was convenient to define the lifts $\hat \Gamma_J$ of the cones $\Gamma_J$ as in \eqref{ends_lift}, using the faces $E_J$. We might as well define $\hat \Gamma_J$ by replacing, in \eqref{ends_lift}, $E_J$ with its closure $\bar E_J$. In this proof we will adopt this latter definition of $\hat \Gamma_J$. The idea is to find a decomposition
\[ \mathcal H = \bigcup_{0 \leq |J| \leq n} H_J \]
so that for every $J$, $\h^{-1}(H_J)$ is homeomorphic to $\hat \Gamma_J$ and for every pair $J_1$, $J_2$ with $J_1 \subset J_2$, $\h^{-1}(H_{J_1}) \cap \h^{-1}(H_{J_2})$ is homeomorphic to $\hat \Gamma_{J_1} \cap \hat \Gamma_{J_2}$.  We will construct a subdivision of this type which is also $G^*$ invariant, i.e. such that
\begin{equation} \label{gsimm_decomp}
    R^*_k  H_{J} = H_{R_kJ} 
\end{equation} 
for all transformations $R^*_k$. 
Let us first define $H_{\emptyset}$. Consider the point
\[ q_0 = \left( \frac{1}{3}, \frac{1}{3}, \frac{1}{3} \right),\] 
which clearly lies on $\mathcal S_0$. Define 
\[ q_k = R^{\ast}_k q_0 \]
so that $q_k \in \mathcal S_k$. Let
\[  H_{\emptyset}= \conv \{ q_0, \ldots, q_3 \}. \]
It is easy to see that $H_{\emptyset}$ is a three dimensional, $G^*$ invariant, simplex contained in $\mathcal H$. Moreover $\h^{-1}(H_{\emptyset})$ is homeomorphic to the closure of the Lagrangian coamoeba $\bar C$ and hence to $\hat \Gamma_{\emptyset}$. 

There is a one to one inclusion reversing correspondence between the faces of $H_{\emptyset}$ and the cones $\Gamma_J$, namely $\Gamma_J$ corresponds to the $(3-|J|)$-dimensional face given by
\begin{equation} \label{face_hj} 
         \conv \{ q_k \}_{k \notin J}.
\end{equation}
Let $L_J$ be the $|J|$-dimensional vector subspace of $N_{\R}$ such that $L_{J}^{\perp}$ is parallel to this face. Clearly $L_{J}^{\perp}$ is transverse to $\Gamma_J$, moreover if $J_1 \subset J_2$ then $L_{J_2}^{\perp} \subset L_{J_1}^{\perp}$ therefore the collection $\{ L_J^{\perp} \} $ defines a system of projections $\{ \x_J \}$ (in the sense of Definition \ref{projFace2}) such that if $J_1 \subset J_2$ then 
\[ \x_{J_1} \circ \x_{J_2} = \x_{J_1}. \]

Let us now define $H_{J}$ when $|J| = 1$. Assume that $J=\{ 1 \}$. In this case 
\[ \x_J(x) = x_1 - \frac{1}{3}x_2 - \frac{1}{3}x_3. \]
Given $t \in \R$, let 
\[ x_t = tu_1 \]
and consider the planes $\x_J^{-1}(x_t)$. These planes are invariant with respect to the transformations $R^*_k$ for all $k \neq 1$, i.e. those that leave $\Gamma_{J}$ fixed. For $t=1/9$, this is the plane containing the face \eqref{face_hj} of $H_{\emptyset}$.  Let us consider the sets 
\[ \x_J^{-1}(x_t) \cap \mathcal H, \quad \text{for} \ t \geq 1/9.\]
\begin{figure}[!ht] 
\begin{center}
\includegraphics{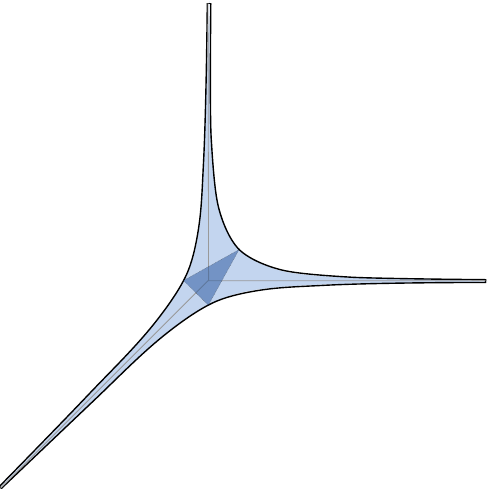}
\caption{The intersection of $\mathcal H$ with the plane $\x_J^{-1}(x_t)$ and the triangle $K_{J,t}$.}
 \label{pants_section_triangle}
\end{center}
\end{figure}
It can be checked that these sets are shaped as in Figure \ref{pants_section_triangle}, i.e. they are similar to the two dimensional version of $\mathcal H$. We denote by $K_{J,t} \subset \x_J^{-1}(x_t) \cap \mathcal H$ the darker triangle in Figure \ref{pants_section_triangle}. It is constructed as follows. One of its vertices is the point 
\[ q_{0,t} = \x_J^{-1}(x_t) \cap \mathcal S_0 \cap \{ x_2 = x_3 \} \]
The other two are
\[ q_{k,t} = R^*_k(q_{0,t}) \]
with $k=2,3$. It is easy to compute that 
\[ q_{0,t}= \left( \frac{2}{3} z(t) + t, z(t), z(t) \right), \]
where $z(t)$ is the unique positive solution of the equation $9z^2(2z+3t)-1=0$. Now,  for $J = \{ 1 \}$, we define 
\[ H_{J} = \bigcup_{t \geq 1/9} K_{J,t}. \]
For the other cones with $|J| =1$ define $H_{J}$ by using the symmetry \eqref{gsimm_decomp}. Clearly $K_{J, 1/9}$ is the face \eqref{face_hj} of $H_{\emptyset}$. For all $t \geq 1/9$, $\h^{-1}(K_{J,t})$ is homeomorphic to the closure of the face $E_{J}$ of $C$. Therefore 
\[ \h^{-1}(H_{J}) \cong \Gamma_J \times \bar E_J \cong \hat \Gamma_J.\]
Moreover 
\[ \h^{-1}(H_{J}) \cap \h^{-1}(H_{\emptyset}) \cong \bar E_J \cong \hat \Gamma_J \cap \hat \Gamma_{\emptyset}.\]
It remains to construct $H_J$ when $|J| =2$. Let $J=\{1,2 \}$. Given $J_1=\{ 1 \} \subset J$, for every $t$ consider the midpoint of the edge of $K_{J_1,t}$ whose vertices are $q_{0,t}$ and $q_{3,t}$. It can be easily checked that this point lies on $\Gamma_J$ and as $t$ varies in the interval $[1/9, + \infty)$ it traces, inside $\Gamma_J$, a curve, which we denote $\tau_1$, whose equation is
\[\tau_1: \quad  x_1 = \frac{1}{108x_2^2} - x_2, \quad 0 \leq x_2 \leq 1/6. \]
Similarly we can consider $J_2= \{ 2 \} \subset J$ and the analogous curve $\tau_2$, whose equation is obtained from the equation of $\tau_1$ by exchanging $x_1$ and $x_2$. The curve 
\[ \tau_J = \tau_1 \cup \tau_2 \]
cuts $\Gamma_J$ in two connected components: one which contains $\Gamma_{J_1}$ and $\Gamma_{J_2}$ and the other which does not. Denote by $Q_{J}$ the closure of latter component (see Figure \ref{region_in_cone}). 
\begin{figure}[!ht] 
\begin{center}
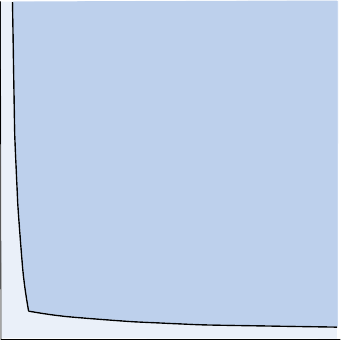
\caption{} \label{region_in_cone}
\end{center}
\end{figure}
It is easy to see that for every $x \in Q_J$, the set 
\[ K_{J,x} = \x_{J}^{-1}(x) \cap \mathcal H \]
is a closed segment with vertices on the surfaces $\mathcal S_0$ and $\mathcal S_3$. Therefore $\h^{-1}(K_{J,x})$ is homeomorphic to a circle, i.e. to $E_J$. Moreover, by construction, if $x \in \tau_1$, then $K_{J,x}$ coincides, for some $t$, with the edge of $K_{J_1,t}$ whose vertices are $q_{0,t}$ and $q_{3,t}$. Similarly if $x \in \tau_2$. In particular, if $x=(1/6, 1/6, 0)= \tau_1 \cap \tau_2$ then $K_{J,x}$ coincides with the edge \eqref{face_hj} of $H_{\emptyset}$. Define 
\[ H_J = \bigcup_{x \in Q_J} K_{J,x}. \]
Then 
\[ \h^{-1}(H_J) \cong Q_J \times E_J \cong \Gamma_J \times E_J \cong \hat \Gamma_J. \]
Moreover, the above observations imply that for all $J'$ with $J' \subset J$ 
\[ \h^{-1}(H_{J}) \cap \h^{-1}(H_{J'}) \cong \Gamma_{J'} \times E_J \cong \hat \Gamma_J \cap \hat \Gamma_{J'}.\]
This concludes the proof. 
\end{proof}

\subsection{The Maslov class of a Lagrangian pair of pants}
Let us recall the definition of the Maslov class of an orientable Lagrangian submanifold $L$ in a Calabi-Yau manifold $X$, with $\dim X = n+1$.  If $\Omega$ is a nowhere vanishing holomorphic $(n+1)$-form on $X$ and $\vol_{L}$ is a volume form on $L$, then there exists a smooth function $\theta: L \rightarrow S^1$ and a positive function $\psi: L \rightarrow \R$ such that 
\[ \Omega|_{L}= \psi e^{i\pi \theta} \vol_{L}. \]
Then $d \theta$ is a closed one form on $L$ and its class in $H^{1}(L, \R)$ is the Maslov class of $L$. In particular $L$ has vanishing Maslov class if and only if we can lift $\theta$ to be an $\R$ valued function. Recall also that when $\theta$ is constant, then $L$ is special Lagrangian of phase $\theta$. 

In our situation $M_{\R} \times T$ is Calabi-Yau and, with respect to complex coordinates $z_j = y_j + i x_j$,  a holomorphic $n+1$-form is given by
\[ \Omega = dz_1 \wedge \ldots \wedge dz_{n+1}. \] 

\begin{ex}
If $\Gamma_J$ is an $n$-dimensional cone of $\Gamma$, then the lift $\hat \Gamma_J$ of $\Gamma$ is special Lagrangian of phase $\theta = \pm 1/2$ (the sign depending on the orientation of $\hat \Gamma_J$). 
\end{ex}

We have the following
\begin{prop} \label{maslov_pants} Lagrangian pairs of pants have vanishing Maslov class.
\end{prop}
\begin{proof} Let $L = \Phi(\tilde C)$ be an $n+1$-dimensional Lagrangian pair of pants and let $\theta: L \rightarrow S^1$ be the function defined above. We have to prove that $d \theta$ is an exact one form on $L$. First observe that 
\[ H_1(L, \Z) \cong \Z^{n+1}. \]
Let $\tilde E_J$ be an edge of $\tilde C$ and let $k \notin J$. In \S \ref{projections} we have proved that there is a neighborhood $\tilde{\mathcal W}_{J,k}$ of $\tilde E_J$ such that the map $\h_{J,k}: \inter \tilde{\mathcal W}_{J,k} \rightarrow \inter \Gamma_J$ is a fibre bundle with fibre $\tilde E_J \cong S^1$. It is easy to see that a basis for $H_1(L, \Z)$ is given by the choice of a fibre of the map $\h_{J,0}$ as $J$ ranges among the sets with $|J|=n$ and $0 \notin J$. On the other hand, Corollary \ref{fbr_bndl} tells us that $\Phi( \inter \tilde{\mathcal W}_{J,0})$ is the graph of an exact one form defined over $\tilde E_J \times \inter \Gamma_J = \inter \hat \Gamma_J$. Since $\hat \Gamma_J$ is special Lagrangian, this implies that $\Phi( \inter \tilde{\mathcal W}_{J,0})$ must have vanishing Maslov class. In particular $d \theta$ is exact on $\Phi( \inter \tilde{\mathcal W}_{J,0})$. By the above description of a basis of $H_1(L, \Z)$, $d \theta$ is exact on $L$. 
\end{proof}

\section{Lifting tropical curves in $\R^2$} \label{lifting_trop}
In this section we restrict to the case $n=1$ and, as an application of Lagrangian pairs of pants, we prove the following theorem

\begin{thm} \label{trop_to_lag}
Let $\dim M_{\R} =2$. Given a smooth tropical curve $\Xi$ in $M_{\R}$ constructed from a unimodal, regular subdivision $(P, \nu)$ of a convex integral lattice polytope $P \subset N_{\R}$, there exists a one parameter family $\mathcal L_t$ of smooth Lagrangian submanifolds of $T^*T = M_{\R} \times N_{\R}/N$ such that 
\begin{itemize}
\item[a)] for all $t$, $\mathcal L_t$ is homeomorphic to the PL lift $\hat \Xi$ of $\Xi$;
\item[b)] as $t \rightarrow 0$, $\mathcal L_t$ converges to $\hat \Xi$ in the Hausdorff topology. 
\end{itemize}
\end{thm}

Before giving the proof, we use the results of \S \ref{projections} to describe the behavior of the embedding $\Phi: \tilde C \rightarrow T^*T$ in a neighborhood of an edge of $\tilde C$.

\subsection{Projection to the legs} \label{proj_legs}  Let $C \subset T$ be a two dimensional Lagrangian coamoeba and let $\Gamma$ be its dual tropical line. Recall that the edges of $C$ are labeled by $E_j$, such that $p_j$ is the unique vertex of $C$ not in $E_j$, i.e. $J= \{j \}$. Similarly, the leg of $\Gamma$ generated by $u_j$ is labeled by $\Gamma_j$. 

For every $k \neq j$, we have the projections $\y_{j,k}$ and $\x_{j,k}$ as in Example \ref{proj_st1}, the associated maps $\h_{j,k}$ and $\gb_{j,k}$ and the neighborhoods $\tilde{\mathcal W}_{j,k}$ of $\tilde E_j$. 

Since $\Gamma_j$ has a unique primitive integral generator $u_j$, it can be canonically identified with $\R_{\geq 0}$ by identifying points of $\Gamma_j$ with their coordinate with respect to $u_j$. Given a point $r_j \in \inter \Gamma_j$ let
\begin{equation} \label{qsets}
 Q_{r_j} = \left \{ x \in \Gamma_j \, | \, x \geq r_j \right \}.
\end{equation}
For every leg $\Gamma_j$, let us fix once an for all a $k_j \neq j$ and denote 
\[ \y_j = \y_{j,k_j}, \quad \x_j = \x_{j, k_j}, \quad \h_j = \h_{j, k_j}, \quad \gb_j = \gb_{j, k_j}. \]
Let
\begin{equation}  \label{endsofH}
     \mathcal H_{r_j} = \mathcal H \cap \x_j^{-1}(Q_{r_j}). 
\end{equation}
Obviously we have that $\h^{-1}(\mathcal H_{r_j}) \subseteq \tilde{\mathcal W}_{j,k_j}$, see Figure \ref{amoeba_level_sets} depicting these sets. In particular Corollary \ref{fbr_bndl} implies
\begin{cor} \label{fbr_bnd_leg}
The map 
\[ \gb_j: \h^{-1}(\mathcal H_{r_j}) \rightarrow Q_{r_j} \times \tilde E_j \]
is a diffeomorphism and $\Phi (\h^{-1}(\mathcal H_{r_j}))$ is the graph of the differential of a function defined on $Q_{r_j} \times \tilde E_j$.
\end{cor}

\begin{figure}[!ht] 
\begin{center}
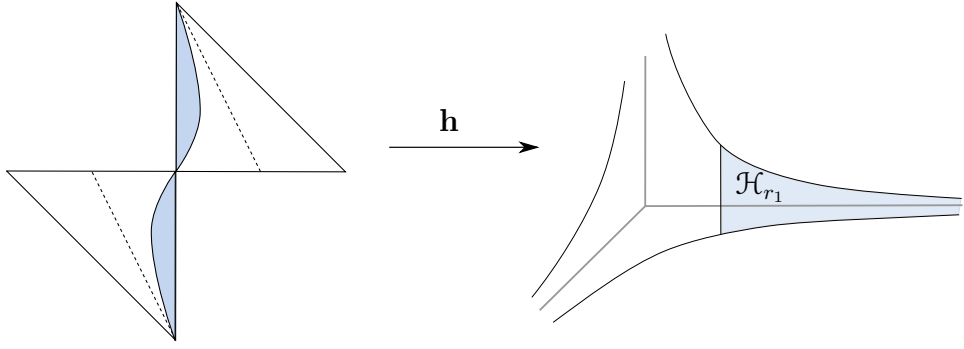
\caption{The shaded area in the coamoeba (left) represents the preimage of $\mathcal H_{r_1}$. The dashed lines delimit the neighborhood $\tilde{\mathcal W}_{1,0}$} \label{amoeba_level_sets}
\end{center}
\end{figure}

For every $j$, choose $r_j \in \Gamma_j$ such that 
\[ \mathcal H_{r_j} \subset \inter \mathcal V_j \]
where the neighborhood $\mathcal V_j$ of $\Gamma_j$ is defined in \eqref{h_nbhoods}.  In particular all the $\mathcal H_{r_j}$'s are pairwise disjoint. Define the set
\begin{equation} \label{trim1}
   \mathcal H^{[1]} = \mathcal H - \bigcup_{j} \mathcal H_{r_j}. 
\end{equation} 
Clearly $\mathcal H^{[1]}$ is homeomorphic to $\mathcal H$ and $\h^{-1}(\mathcal H^{[1]})$ is homeomorphic to a pair pants, in particular to $\tilde C$.

\subsection{Rescaling} \label{resc} It is convenient to consider also a rescaled Lagrangian pairs of pants $\Phi_{\lambda}(\tilde C)$, constructed replacing the function $F$ with the function $F_{\lambda} = \lambda F$. We continue to denote by $\mathcal H$ the image of $\h_{\lambda}$, by $\mathcal H_{r_j}$ the subsets \eqref{endsofH} and by $\mathcal H^{[1]}$ the subset \eqref{trim1}. We have the following
\begin{lem} \label{resc_small} Let $B \subset M_{\R}$ be a neighborhood of the origin. For every $j = 0, 1, 2$ choose points $r_j \in \inter \Gamma_j$ such that 
\[ \conv \{ r_0, r_1, r_2 \} \subset B. \]
Then there exists a rescaled Lagrangian pair of pants $\Phi_{\lambda}: \tilde C \rightarrow T^*T$ such that for all $j$
\[ \mathcal H_{r_j} \subset \inter \mathcal V_j \] 
and 
\[ \mathcal H^{[1]} \subset B. \]
\end{lem}
This easily follows from the definition and properties of $\Phi_{\lambda}$ and $\mathcal H$. 

\subsection{Proof of Theorem \ref{trop_to_lag}} \label{proof_mainthm}
{\it Step 1: the local models.}
Let $\check e$ be a vertex of $\Xi$. Recall definition \eqref{star_neigh} of the star-neighborhood $\Xi_{\check e}$. Define the tangent tropical line $\Gamma_e \subseteq M_{\R}$ of $\Xi$ to be the cone of $\Xi_{\check e}$  with center $\check e$, i.e. 
 \begin{equation} \label{tangent_tropical_line}
           \Gamma_e = \{ \check e + t(v-\check e) \in M_{\R} \, | \, v \in \Xi_{\check e} \ \text{and} \ t \in \R_{\geq 0} \}.
 \end{equation}
Notice that $\check e$ is the vertex of $\Gamma_e$. This tropical line is dual to the coamoeba $C_e$ defined in \S \ref{LagPLift}. In particular there is a one to one correspondence between edges $C_f$ of $C_e$ and legs of $\Gamma_e$ which we denote by
\[ \check f \mapsto \Gamma_{e,f}.\]
Also denote by $V_f \subset M_{\R}$ the affine line containing $\Gamma_{e,f}$.

There are natural coordinates $x=(x_1, x_2)$ adapted to a tangent tropical line $\Gamma_e$ given by choosing $\check e$ as the origin and a basis $\{ u_1, u_{2} \}$ of $M$ such that each $u_j$ is the primitive integral generator of a one dimensional cone of $\Gamma_e$. With respect to these coordinates $\Gamma_e$ is identified with the standard tropical line $\Gamma$. Dually, let $\{u^*_1, u^*_{2} \}$ be a basis of $N_{\R}$ satisfying \eqref{dual_basis}. Then this basis and the choice of a vertex of $C_e$ as the origin of $T$ defines coordinates $y=(y_1, y_{2})$, with respect to which $C_e$ is identified with the standard Lagrangian coamoeba $C$. It is clear that such a choice of coordinates is unique up to a transformation in the group $G^*$ and in its dual $G$. In particular we can view the function $F$ of \eqref{Fglob} as being defined on $\tilde C_e$ and thus we have a Lagrangian pair of pants
\begin{equation} \label{local_vert}
   \begin{split} 
     \Phi_e: \check e \times \tilde C_e & \rightarrow M_{\R} \times N_{\R} / N  \\
                                           y             & \mapsto (\check e + (dF)_y, y).
   \end{split}    
\end{equation}
Denote by 
\[ \h_e:  \check e \times \tilde C_e \rightarrow M_{\R} \]
the composition of $\Phi_e$ with the projection onto $M_{\R}$ and by $\mathcal H_{e}$ the image of $\h_e$. In the local coordinates $(x,y)$, $\Phi_e$, $\h_e$ and $\mathcal H_e$ coincide with $\Phi$, $\h$ and $\mathcal H$ respectively. 

For every two dimensional $e \in (P, \nu)$ and every edge $f \preceq e$, let $p_{e,f} \in C_e$ be the unique vertex of $C_e$ which is not in $C_f$ and let $U_{e,f} = C_{e} - p_{e,f}$. Then choose and fix projections $\x_{e,f}: M_{\R} \rightarrow V_f$ and $\y_{e,f}: \tilde U_{e,f} \rightarrow C_{f}$ as in Example \ref{proj_st1} (see also \S \ref{proj_legs}). Thus we have the map $\gb_{e,f}: \tilde U_{e,f} \rightarrow V_f \times C_f$. 

For any choice of points $r_{e,f} \in \inter \Gamma_{e,f}$, we define $Q_{r_{e,f}} \subset \Gamma_{e,f}$ and $\mathcal H_{r_{e,f}} \subset \mathcal H_e$ as in \eqref{qsets} and \eqref{endsofH} respectively. 
We then have that  
\[ \gb_{e,f}: \h_e^{-1}(\mathcal H_{r_{e,f}}) \rightarrow Q_{r_{e,f}} \times C_f \]
is a diffeomorphism by Corollary \ref{fbr_bnd_leg}. 

As such, the image of $\Phi_e$ may be too big for our purposes. So we need to rescale it. For every vertex $\check e$ of $\Xi$ fix an open convex neighborhood $B_e$ of $\check e$ such that all the $B_e$'s are pairwise disjoint. Moreover, on every leg $\Gamma_{e,f}$ of $\Gamma_e$, choose a pair of points  $r'_{e,f}<r_{e,f}$ such that $r'_{e,f}, r_{e,f} \in B_e$.  By applying Lemma \ref{resc_small}, we can suitably rescale the Lagrangian pair of pants \eqref{local_vert} so that 
\[ \mathcal H_{r_{e,f}} \subset \mathcal H_{r'_{e,f}} \subset \inter \mathcal  V_{e,f}. \]
In particular we can define the subset 
\[ \mathcal H^{[1]}_e = \mathcal H_e - \bigcup_{f \preceq e} \mathcal H_{r_{e,f}} \] 
and assume that 
\begin{equation} \label{trim_in_B}
  \mathcal H^{[1]}_e  \subset B_e.
\end{equation}

\medskip

{\it Step 2: the gluing.} We now glue together the local models. Denote by $G_{e,f}: Q_{r'_{e,f}} \times \tilde C_f \rightarrow \R$ the Legendre transform of $F$ constructed in Proposition \ref{faceproj}, so that the graph of $dG_{e,f}$ coincides with $\Phi_e(\h_e^{-1}(\mathcal H_{r'_{e,f}}))$.
Let  $[r'_{e,f}, r_{e,f}] \subset \Gamma_{e,f}$ denote the segment joining $r'_{e,f}$ and $r_{e,f}$. Given two points $r''_{e,f}, \bar r_{e,f} \in [r'_{e,f}, r_{e,f}]$ such that $r'_{e,f} < r''_{e,f} < \bar r_{e,f} < r_{e,f}$, let $\eta: [r'_{e,f}, r_{e,f}] \rightarrow \R$ be some smooth, non-increasing function such that
\[ \eta(x) = \begin{cases}
                         1 \quad x \in [r'_{e,f}, r''_{e,f}],          \\
                         0 \quad x \in [\bar r_{e,f} , r_{e,f}]. 
                 \end{cases}
 \]
Define the following smooth function $G'_{e,f}: [r'_{e,f}, r_{e,f}] \times C_f \rightarrow \R$
 \[ G'_{e,f}(x, y)= \eta(x) G_{e,f}(x,y)  \]
We use it to define a Lagrangian embedding as the graph of $dG'_{e,f}$
\begin{equation}
 \begin{split}
         \Phi_{e,f}: [r'_{e,f}, r_{e,f}] \times C_f & \rightarrow  M_{\R} \times N_{\R}/N \\
                                 (x,y) & \mapsto (x,y, (dG'_{e,f})_{(x,y)}) 
 \end{split}
\end{equation}
where here we use the immersion of the cotangent bundle of  $[r'_{e,f}, r_{e,f}] \times C_f$ in $M_{\R} \times N_{\R}/N$ given by Lemma \ref{cotangent}. By construction we have that 
\begin{equation} \label{overlap}
 \Phi_{e,f}( [r'_{e,f}, r''_{e,f}) \times C_f) = \Phi_e( \h_e^{-1}(\mathcal H_{r'_{e,f}} - \mathcal H_{r''_{e,f}}))
\end{equation}
and 
\begin{equation} \label{flattens}
   \Phi_{e,f}( [\bar r_{e,f}, r_{e,f}] \times C_f) =  [\bar r_{e,f}, r_{e,f}] \times C_f.
\end{equation}

\medskip

If $\check f$ has two vertices $\check e_1$ and $\check e_2$, let $\rho'_f \subset \check f$ (respectively $\rho_f$) be the segment joining the points $r'_{e_1,f}$ and $r'_{e_2,f}$ (resp. $r_{e_1,f}$ and $r_{e_2,f}$) . Otherwise if $\check f$ is an unbounded edge with only one vertex $\check e$, let $\rho'_{f}$  (resp. $\rho_f$) be the ray inside $\check f$ which starts at $r'_{e,f}$ (resp. at $r_{e,f}$). We have $\rho_f \subset \rho'_f$.  Now extend the Lagrangian embedding $\Phi_{e,f}$ to a Lagrangian embedding 
\[ \Phi_f: \rho'_f \times C_f \rightarrow M_{\R} \times N_{\R}/N \]
by defining $\Phi_f$ to be equal to $\Phi_{e,f}$ on $[r'_{e,f}, r_{e,f}] \times C_f$ for every vertex $\check e$ of $\check f$ and to be the inclusion on $\rho_f \times C_f$. Clearly $\Phi_f$ is a smooth extension by \eqref{flattens}. Also define 
\[ \h_f:  \rho'_f \times C_f \rightarrow M_{\R}  \]
to be equal to $\Phi_f$ followed by projection onto $M_{\R}$. 

\medskip 

{\it Step 3: the construction of $\mathcal L_t$.} We now put all the pieces together. For every vertex $\check e$ of $\Xi$, define the following open subset of $\tilde C_e$
\[ \mathcal Z_e = \h_e^{-1} \left( \mathcal H_e - \bigcup_{f \preceq e} \mathcal H_{r''_{e,f}} \right) \]
and the Lagrangian manifold  
\[ \mathcal L_e = \Phi_e(\mathcal Z_e). \]
For every edge $\check f$ of $\Xi$ define the Lagrangian submanifold
\[ \mathcal L_f = \Phi_f(\rho'_f \times C_f) \]
and then take the union 
\[ \mathcal L = \left( \bigcup_{\check e} \mathcal L_e \right) \cup  \left( \bigcup_{\check f} \mathcal L_f \right) \]
We need to show that $\mathcal L$ is a smooth submanifold. This is done by showing that the only intersections between the various pieces is given by \eqref{overlap} when $\check e$ is a vertex of $\check f$. Given two vertices $\check e_1$ and $\check e_2$, clearly 
\[ \mathcal L_{e_1} \cap \mathcal L_{e_2} = \emptyset \]
since the projection of $\mathcal L_{e_j}$ to $M_{\R}$ is contained in $B_{e_j}$ by \eqref{trim_in_B} and $B_{e_1} \cap B_{e_2} = \emptyset$.

\begin{lem}
Given two edges $\check f_1$ and $\check f_2$, we have
\[  \mathcal L_{f_1} \cap \mathcal L_{f_2} = \emptyset.\]
Given a vertex $\check e$ and an edge $\check f$, then $\mathcal L_{e} \cap \mathcal L_{f}$ is non empty if and only if $\check e$ is a vertex of $\check f$. In which case 
\[ 
  \mathcal L_{e} \cap \mathcal L_{f} = \Phi_{f}( [r'_{e,f}, r''_{e,f}) \times C_f) = \Phi_e( \h_e^{-1}(\mathcal H_{r'_{e,f}} - \mathcal H_{r''_{e,f}})).
 \]
\end{lem}
\begin{proof}
It is clear that $\Phi_f(\rho_f \times C_f)$ is disjoint from any other piece, since its projection to $M_{\R}$ coincides with $\rho_f$ which does not intersect the projection of any other piece. The rest of the statement follows from the observation that 
\[ \h_f([r''_{e,f}, r_{e,f}] \times C_f) \subseteq \mathcal H_{r''_{e,f}} - \mathcal H_{r_{e,f}}. \]
 Let us prove this inclusion. 

By using coordinates $(x,y)$ on $M_{\R} \times N_{\R}/N$ adapted to $\Gamma_e$ as above, we can assume $\Gamma_{e,f} = \Gamma_2$, $\x_{e,f}=\x_{2,0}$ and $\y_{e,f}=\y_{2,0}$. Therefore $(x_2, y_1)$ are coordinates on $[r'_{e,f}, r_{e,f}] \times C_f$ and it can be easily seen that 
\[ \h_f(x_2,y_1) = \left( \frac{\partial G'_{e,f}}{\partial y_1}, x_2 \right) = \left( \eta(x_2) \frac{\partial G_{e,f}}{\partial y_1}, x_2 \right). \]
In particular, by definition of $G_{e,f}$ and since $0 \leq \eta(x_2) \leq 1$, this implies that for all $x_2 \in [r''_{e,f}, r_{e,f}]$ we have
\[ \h_f \left( \{ x_2 \} \times C_f \right) \subseteq \x_{e,f}^{-1}(x_2) \cap \mathcal H_e \subseteq \mathcal H_{r''_{e,f}} - \mathcal H_{r_{e,f}}. \]
This concludes the proof of the lemma.
\end{proof}

This Lemma implies then that $\mathcal L$ is a smooth manifold. It is obvious that $\mathcal L$ is also homeomorphic to $\hat \Xi$. To construct the family $\mathcal L_t$ it is enough to rescale all local models $\Phi_e$ and $\Phi_f$ uniformly by a factor of $t$, with $t \in [0,1]$. The fact that $\mathcal L_t$ converges in the Hausdorff topology to $\hat \Xi$ is clear from the construction (see also Corollary \ref{hausdorf_conv}). This concludes the proof of Theorem \ref{trop_to_lag}. 

\section{Some generalizations} \label{generalizations}
Here we explain three different generalizations of the above result. We do not yet have a complete understanding of these constructions, so we will mainly discuss some examples and pose some questions. 

\subsection{Different lifts of the same tropical curve?} \label{lifts_same_crv}
Here assume that $\dim M_{\R}=2$. If we compactify $X = M_{\R} \times T$ at infinity by adding a boundary $B \cong S^1 \times T$ (i.e. viewing $M_{\R}$ as a disk whose boundary is the first factor of $B$) then a Lagrangian lift $\mathcal L$ of $\Xi$ is a manifold with boundary on $B$. Suppose we have two lifts $\mathcal L_1$ and $\mathcal L_2$ of the same tropical curve $\Xi$ such that 
\[ \partial \mathcal L_1 = \partial \mathcal L_2. \]
Then $\mathcal L_2 - \mathcal L_1$ defines a homology class in $H_2(X, \Z) \cong H_2(T, \Z) = \Z$. It is not hard to show that one can construct pairs of Lagrangian lifts such that their difference represents any homology class. Let $\mathcal L_1$ be the lift constructed in previous sections, based on the $PL$ lift $\hat \Xi$. 
Let $\check f$ be an edge of $\Xi$, denote by $M^f_{\R}$ the tangent line of $\check f$ ($M^f$ is the lattice spanned by the integral tangent vectors) and by $(M^{f}_{\R})^{\perp} \subset N_{\R}$ its orthogonal space. Now consider the quotient 
\begin{equation}  \label{quotient_edge}
     \alpha: \check f \times T \rightarrow \check f \times \frac{T}{(M^f_{\R})^{\perp}}
\end{equation}
and a smooth section 
\begin{equation} \label{section_edge}
        \sigma_f : \check f \rightarrow \check f \times \frac{T}{(M^f_{\R})^{\perp}}.
\end{equation}
Define
\begin{equation} \label{general_lift_edge}
       \hat f_{\sigma_f} = \alpha^{-1}(\sigma_f(\check f)),
\end{equation}
then it is easy to see that this is still a Lagrangian submanifold. Notice that the canonical lift $\hat f$, defined in \S \ref{LagPLift}, corresponds to a constant section, indeed $\alpha(C_f)$ is a point in $T/(M^f_{\R})^{\perp}$.  We have that $T/(M^f_{\R})^{\perp} \cong S^1$. For every edge $\check f$, let $\sigma_f$ be a section such that its projection to $T/(M^f_{\R})^{\perp}$ defines a loop with base point $\alpha(C_f)$. Denote by $\sigma$ the collection $\{ \sigma_f\}_{\dim \check f = 1}$. Then define the twisted PL lift to be
\[ \hat \Xi_{\sigma} = \left( \bigcup_{\dim e=2} \check{e} \times C_e \right) \cup \left( \bigcup_{\dim f =1} \hat f_{\sigma_f} \right). \] 
Then $\hat \Xi_{\sigma}$ is still a topological submanifold, Lagrangian on its smooth part. If $\tau$ is a generator of $H_2(X, \Z)$ (represented by the torus $T$) then as a homology class
\[ \hat \Xi_{\sigma}- \hat \Xi = n_{\sigma} \tau \]
for some integer $n_\sigma$. In fact $n_{\sigma}$ can be written as
\[ n_{\sigma} = \sum_{f} n_{\sigma_f}, \]
where $n_{\sigma_f}$ is the class of the loop defined by $\sigma_f$ in $H_1(T/(M^f_{\R})^{\perp}, \Z) \cong \Z$. 

It is easy to see that in the same way we can also construct a smooth lift $\mathcal L_{\sigma}$. In fact recall that the lift $\mathcal L$, over the segment $\rho_f \subset \check f$ defined in \S \ref{proof_mainthm}, coincides with $\rho_f \times C_f$. Therefore it is enough to consider a section $\sigma_f$ which is constant and equal to the point $\alpha(C_f)$ outside $\rho_f$ and replace $\rho_f \times C_f$ with $\alpha^{-1}( \sigma_f(\rho_f))$. Denote the corresponding lift by $\mathcal L_{\sigma}$. 

\begin{rem} What is the relationship between $\mathcal L$ and $\mathcal L_{\sigma}$? Are they Hamiltonian equivalent? Of course, if $\mathcal L- \mathcal L_{\sigma}$ represents a non-trivial class, then there is no compactly supported Hamiltonian diffeomorphism between them, but there may be some non compactly supported ones, which also move the boundary and unwinds all the twists.
\end{rem}

\begin{rem} We expect that different Lagrangian lifts of the same tropical variety should have some relevance in the context of Homological Mirror Symmetry, as different lifts should be mirror to different coherent sheaves supported on the same complex subvariety corresponding to the tropical variety. This idea is investigated in Section 6.3 of \cite{diri_branes} when $\Xi$ is an affine subvariety of an affine manifold.  An explicit correspondence in the case of toric Calabi-Yau threefolds was proposed in \cite{onHmstoric}. 
\end{rem}

\subsection{Lifting non-smooth tropical curves} \label{non_smooth}
What happens if the tropical curve is not smooth? Can we still construct a Lagrangian lift? We believe it should be possible, although we do not have a complete theory yet. Here we only give some examples. Non smooth tropical curves involve the following generalizations: edges have a weight, i.e. a positive integer $w_f$ attached to each edge $\check f$; vertices may have valency higher than $3$; the primitive integral tangent vectors to the edges emanating from a vertex generate a sublattice of $M$. 

\begin{ex} \label{non_smooth_trivalent} Let $\Xi$ consist of just one $3$-valent vertex $\check e$, with tree edges $\check f_0, \check f_1, \check f_2$ emanating from it, with weights $w_0, w_1, w_2$ respectively. It can be obtained from a general lattice simplex $P$ in $N_{\R}$ (i.e. not necessarily elementary) via the procedure of \S \ref{trop_hyper} using a constant function $\nu$. The edges $\check f_j$ are then dual to the edges $f_j$ of $P$. The weight $w_j$ is the integral length of $f_j$. Of course we also have the balancing condition
\[ \sum w_j u_j= 0 \]
where the vectors $u_j$ are the primitive integral generators of the edges $\check f_j$ (pointing outward from the vertex). In this case it is easy to construct a Lagrangian lift of $\Xi$ as follows. Consider the integer
\[ m_e = | w_1 u_1 \wedge w_2 u_2 | \] 
Let $M' \subset M$ be the lattice 
\[ M'=  \spn_{\Z} \{ w_1u_1, w_2 u_2 \}.\]
It is a sublattice of $M$ of index $m_e$. The dual $N' \subset N_{\R}$ of $M'$ is a lattice which contains $N$ as a sublattice of index $m_e$. 
In particular, if we denote
\[ T' = \frac{N_{\R}}{N'} \]
we have that the action of $N'$ on $T$ defines a degree $m_e$ covering map
\[ \beta: T \rightarrow T'.\]
Observe that $\Xi$, as a tropical curve in $M'_{\R}$, is a standard tropical line. Therefore we know how to construct the Lagrangian coamoeba $C'_e \subset T'$ dual to $\Xi$, moreover we also have the function $F': C'_e \rightarrow \R$ defined in  \eqref{Fglob}. Now define
\[ C_e = \beta^{-1}(C'_e) \]
and the function 
\[ F= F' \circ \beta \] 
on $C_e$. We define the Lagrangian lift of $\Xi$ to be the graph of the differential of $F$ extended to the real blow up of $C_e$ at its vertices. 
Notice that if an edge $\check f$ has multiplicity $w$, the coamoeba $C_e$ will contain $w$ (parallel) copies of the face $C'_f$ of $C'_e$. In particular, to construct a PL lift $\hat \Xi$ of $\Xi$ we must attach $w$ copies of $\check f \times C'_f$ to $C_e$. As in Theorem \ref{trop_to_lag}, by rescaling $F$ by a factor of $t$ we can obtain a one parameter family of smooth lifts which converge to $\hat \Xi$.
\begin{figure}[!ht] 
\begin{center}
\includegraphics{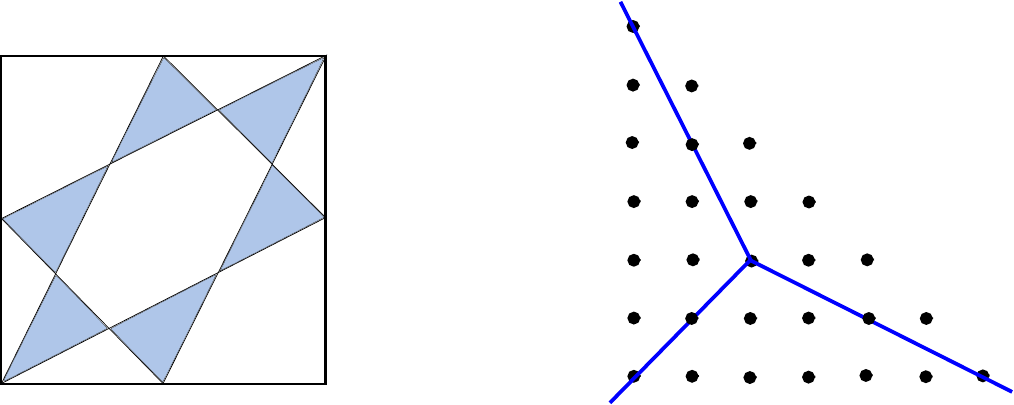}
\caption{} \label{co_am_genus1}
\end{center}
\end{figure}
\end{ex}
As an example consider the tropical curve in Figure \ref{co_am_genus1}. It is not smooth, but its edges have multiplicity one. The associated coamoeba is depicted on the left. Notice that the lift is a torus with three punctures.  

\begin{ex} \label{with_mult}Let $\Xi$ be the standard tropical line, but let us assign to each edge of $\Xi$ a multiplicity $w$. Then $m_e = w^2$.  Figure \ref{co_am_multipl} shows the co-ameba for the case $w=2$
\begin{figure}[!ht] 
\begin{center}
\includegraphics{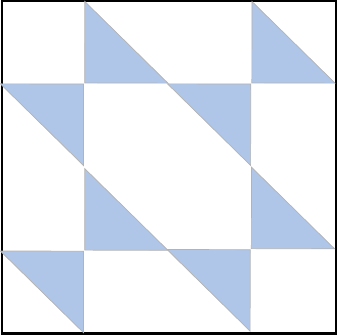}
\caption{} \label{co_am_multipl}
\end{center}
\end{figure}
We have that the lift is a surface of genus $g=\frac{(w-1)(w-2)}{2}$ with $3w$ punctures. As expected, the formula for the genus is the same as for a complex curve of degree $w$ in $\PP^2$. 
\end{ex}

\begin{ex} \label{four_val} Here we give an example of a $4$-valent vertex. Let 
\[ P = \conv \{ (1,0), (0,1), (-1,0), (0,-1) \} \]
and let $\Xi$ be the tropical curve associated to $P$ with the function $\nu =0$. Then $\Xi$ has one vertex and four (unweighted) edges  $\check f_0, \ldots, \check f_3$ generated respectively by the vectors $u_0 = (1,1)$, $u_1=(-1,1)$, $u_2=(-1,-1)$ and $u_3 = (1,-1)$ (see Figure \ref{co_am_4valent}). Represent the torus $T$ as in \eqref{toro} and let $C_e$ be the subset of $T$ depicted in Figure \ref{co_am_4valent} (on the left). 
\begin{figure}[!ht] 
\begin{center}
\includegraphics{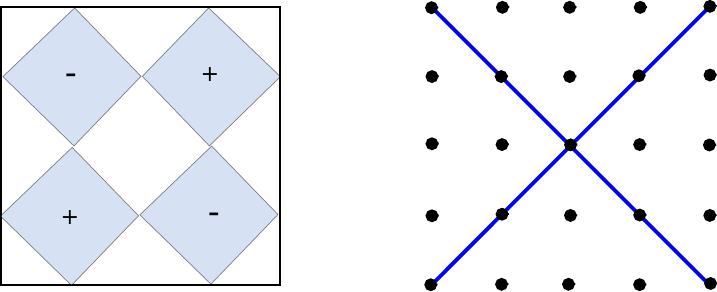}
\caption{} \label{co_am_4valent}
\end{center}
\end{figure}
The edges of the four squares forming $C_e$ are contained in the lines defined by equations
\[ \inn{u_j}{y}= \frac{\pi}{4} \quad \mod \pi \Z \]
Mark each square with a plus or minus sign as in the figure and let $\epsilon(y)$ be equal to $1$ or $-1$ if $y$ is respectively in a plus or minus square. Define the function 
\[ F(y) = \epsilon(y) \left( \prod_{j=0}^{3} \left| \sin \left (\inn{u_j}{y} - \frac{\pi}{4} \right) \right| \right)^{1/2} \]
Then $F$ is well defined on $C_e$ and its behavior near a vertex is similar to the behavior of the function $F$ defined in \eqref{Fglob}. Therefore the graph of the differential of $F$ extends to a Lagrangian embedding of the real blow up of $C_e$ at its vertices. This it the Lagrangian lift of $\Xi$. 
\end{ex}

\subsection{Lifting tropical curves in $\R^n$}

We sketch here how one can construct Lagrangian lifts of smooth tropical curves in higher codimensions, i.e. in $\R^n$ with $n \geq 3$. See also Mikhalkin's paper \cite{mikh_trop_to_lag}. So let $\Xi$ be a smooth tropical curve in $M_{\R}$, with $M \cong \Z^n$. This means that all vertices are trivalent and that any two primitive, integral tangent vectors to the edges emanating from a vertex are part of a basis of $M$. The balancing condition implies that all edges emanating from a vertex $v$ of $\Xi$ are contained in an affine plane, which we denote by $V_v$. The direction of $V_v$ is a two dimensional vector subspace of $M_{\R}$ which we denote $M^{v}_{\R}$, where $M^v$ is the lattice spanned by the integral tangent vectors to the edges emanating from $v$. Similarly, if $\ell$ is an edge of $\Xi$, denote by $V_{\ell}$ the affine line containing $\ell$ and by $M^{\ell}_{\R}$ its direction. To define a Lagrangian lift of $\ell$ inside $M_{\R} \times T$, we can proceed as in \S \ref{lifts_same_crv}, i.e. we consider the quotient as in \eqref{quotient_edge}, then we take a section $\sigma_{\ell}$ as in \eqref{section_edge} and define $\hat \ell_{\sigma_{\ell}} \subset M_{\R} \times T$ as in \eqref{general_lift_edge}, i.e. as the preimage of the section by the quotient map. Notice that $\hat \ell_{\sigma_{\ell}}$ is a Lagrangian submanifold. 

Now let us consider a vertex $v$. Notice that the dual of the lattice $M^{v}$ can be naturally identified with the lattice
\[ N^v = \frac{N}{(M^v_{\R})^{\perp} \cap N}. \]
Moreover if $\Gamma_v$ is the tangent tropical line at $v$ (see \eqref{tangent_tropical_line}), i.e. the cone spanned by the edges emanating from $v$, we have that $\Gamma_v$ is a standard tropical line in $M^{v}_{\R}$ and thus defines a dual Lagrangian coamoeba $C_v$ in the torus
\[ T^v = \frac{N^{v}_{\R}}{N^{v}}. \]
Notice that we naturally have
\[ T^v \cong  \frac{T}{(M^{v}_{\R})^{\perp}}. \]
Therefore if we consider the quotient
\[ v \times T \rightarrow v \times \frac{T}{(M^{v}_{\R})^{\perp}}, \]
we define $\hat v$ to be the preimage of $C_v$ via this quotient. Then we define the PL lift of $\Xi$, depending on the collection $\sigma$ of the sections $\sigma_{\ell}$, to be 
\[ \hat \Xi_{\sigma} = \left( \bigcup_{\dim v =0} \hat v \right) \cup \left( \bigcup_{\dim \ell =1} \hat \ell_{\sigma_\ell} \right). \] 
For $\Xi_{\sigma}$ to be a topological submanifold, $\hat v$ has to glue nicely to $\hat \ell_{\sigma_\ell}$ for all edges $\ell$ emanating from $v$. This can be done as follows. Since $(M^v_{\R})^{\perp} \subset (M^{\ell}_{\R})^{\perp}$, the latter space acts on $T^v$ and we have
\[ \frac{T^v}{(M^{\ell}_{\R})^{\perp}} \cong \frac{T}{(M^{\ell}_{\R})^{\perp}}. \]
Moreover, if $C_{v, \ell}$ is the edge of $C_v$ corresponding to the edge $\ell$, the image of $C_{v, \ell}$ in the latter quotient is a point. We require $\sigma_{\ell}$ to be equal to this point at  $v$. This condition ensures that $\hat \Xi_{\sigma}$ is a topological submanifold. 

The smoothing of $\hat \Xi_{\sigma}$ can be done exactly as in the proof of Theorem \ref{trop_to_lag}. We replace $C_v$ by a Lagrangian pair of pants inside $M^{v}_{\R} \times T^v$, suitably deformed so to coincide with the lifts of the edges away from $v$.  Then we take the preimage of this Lagrangian pair of pants via the quotient map $M^v_{\R} \times T \rightarrow M^{v}_{\R} \times T^v$. 

\subsection{Exact examples}
Here we investigate the exactness condition for the Lagrangian lifts of tropical curves. We prove the following 
\begin{thm}
If a Lagrangian lift in $T^*T$ of a tropical curve $\Xi$ in $M_{\R}$ is exact then all edges of $\Xi$ lie on one dimensional vector subspaces of $M_{\R}$. In particular $\Xi$ can have only one vertex in the origin. 
\end{thm}

\begin{proof} In the coordinates $(x,y)$ on $T^*T$ given in \S \ref{setup}, the canonical one form on $T^*T$ is 
\[ \lambda = \sum_{i} x_i dy_i. \]
In the group $H_{1}(\hat \Xi, \Z)$ consider the classes of the loops $\{ x \} \times C_f$ where $x$ is some point on the edge $\check f$ of $\Xi$ and $\hat f = \check f \times C_f$. These can be also considered as classes in $H_1(\mathcal L, \Z)$, since if $x \in \rho_f$ then $\{ x \} \times C_f \subset \mathcal L$. Now notice that $C_f$ is an affine line in $T$ with primitive integral tangent vector $e_f$ orthogonal to $\check f$. In particular $\check f$ lies on an affine line in $M_{\R}$ with equation
\[ \inn{e_f}{x} = c_f \]
for some $c_f \in \R$. Therefore we have that 
\[ \int_{\{ x \} \times C_f}  \lambda = c_f. \]
For $\mathcal L$ to be exact, this integral must be zero,  i.e. $c_f = 0$,  for all edges $\check f$ of $\Xi$. In particular all edges lie on one dimensional vector subspaces of $M_{\R}$. 
\end{proof}

Therefore the only smooth tropical curves whose lifts are exact are the tropical lines. The lifts of the non smooth tropical curves in Examples \ref{non_smooth_trivalent}, \ref{with_mult} and \ref{four_val} are all exact, provided the vertex is the origin. 

\section{Lagrangian submanifolds in toric varieties} \label{intoric}
Here we discuss various examples of Lagrangian submanifolds of toric varieties obtained as lifts of tropical curves inside the moment polytope. The idea to construct surfaces in toric varieties (and in almost-toric ones) from curves in the moment polytope can be traced back to the work of Symington, see Definition 7.3 and Theorem 7.4 of \cite{sym_at}, who calls such surfaces visible surfaces. Symington describes how the surface also depends on the interaction between the curve and the boundary of the moment polytope, which is what we will do here too. The richest and most straight forward family consists of Lagrangian submanifolds with boundary on the toric divisor. Then we discuss a new construction, which we learned from Mikhalkin \cite{mikh_trop_to_lag}, of non-orientable Lagrangian surfaces inside $\C^2$. Finally we give examples of Lagrangian tori in $\PP^2$ or $\PP^1 \times \PP^1$, including some monotone ones. 

Let us first recall a few facts about toric varieties. Let $\dim M_{\R} = 2$. Let $\Delta \subset M_{\R}$ be a Delzant polygon. This means that $\Delta$ is an intersection of half spaces, its edges have rational slope and, at each vertex, the primitive integral vectors tangent to the edges adjacent to the vertex form a basis of $M$. Notice that $\Delta$ does not need to be compact.  We denote by $\partial \Delta$ the boundary of $\Delta$ and by $\Delta^{o}$ its interior. In our context, it is useful to think of the toric variety $X_{\Delta}$ with moment polygon $\Delta$ in the following way. As usual $T = N_{\R} / N$. Consider
\[ \hat \Delta = \Delta \times T \]
Then $X_{\Delta}$ is obtained from $\hat \Delta$ by collapsing the fibres over the boundary $\partial \Delta$ as follows. If $d$ is the interior of an edge, let $M^d_{\R}$ be its tangent subspace. Then $X_{\Delta}$ is obtained by replacing $d \times T$ with $d \times (T/(M^d_{\R})^{\perp})$. At a vertex $v$ we replace $v \times T$ by $v$. We denote by 
\[ \mu: X_{\Delta} \rightarrow \Delta \]
the moment map, i.e. the projection onto $\Delta$. Then $\Delta^o \times T \subset X_{\Delta}$ and the symplectic form of $X_{\Delta}$ restricted to this open set coincides with the usual one.

In the following by a tropical curve $\Xi$ inside $\Delta$ we mean that 
\[ \Xi = \Delta \cap \Xi^{\infty} \]
where $\Xi^{\infty}$ is a tropical curve in $M_{\R}$, none of whose vertices are on $\partial \Delta$. 

\subsection{Lagrangian submanifolds with boundary} \label{lag_boundary}
Let $\Xi$ be a tropical curve in $\Delta$ which satisfies the following property:  
\begin{itemize}
\item[(P)] when an edge $\check f$ of $\Xi$ ends on an edge $d$ of $\Delta$, then the primitive integral tangent vectors to $d$ and $\check f$ form a basis of $M$. In particular if an edge of $\Xi$ ends on a vertex of $\Delta$, then it must satisfy this property with respect to both edges adjacent to this vertex.
\end{itemize}
\begin{thm} \label{lagran_bnd}
If $\Xi \subset \Delta$ is tropical curve satisfying (P), then we can lift $\Xi$ to a Lagrangian submanifold with boundary $\mathcal L$ in $X_{\Delta}$. The boundary of $\mathcal L$ is contained in the toric boundary of $X_{\Delta}$ and consists of the union of the circles $\mu^{-1}(p)$ for all points $p$ of $\Xi \cap \partial \Delta$ which are in the interior of an edge $d$ of $\Delta$. If $v \in \Xi \cap \partial \Delta$ is a vertex of $\Delta$ then $\mu^{-1}(v)$ is a (smooth) interior point of $\mathcal L$. 
\end{thm}

\begin{proof} 
Let $\Xi^{\infty}$ be the tropical curve in $M_{\R}$ such that $\Xi = \Delta \cap \Xi^{\infty}$ and let $\mathcal L^{\infty}$ be the lift of $\Xi^{\infty}$ constructed in Section \ref{lifting_trop}. As can be seen from the construction, we can assume that if $p \in \partial \Delta \cap \Xi$ and $\check f$ is the edge of $\Xi$ which contains $p$, then $p \in \rho_f \subset \check f$, i.e. $\rho_f \times C_f \subset \mathcal L$. We now define $\mathcal L$ to be the closure of $\mathcal L^{\infty} \cap (\Delta^o \times T)$ inside $X_{\Delta}$. Equivalently we may consider $\mathcal L$ as the image of $\mathcal L^{\infty} \cap (\Delta \times T)$ after the collapsing of the boundary described above. We have to prove that $\mathcal L$ satisfies the statement of the theorem. 

Notice that by the above observation, $\mathcal L^{\infty} \cap (\partial \Delta \times T)$ consists of the union of the circles $p \times C_f$ for all $p \in \partial \Delta \cap \Xi$ and edges $\check f$ with $p \in \check f$. Suppose $p$ is in the interior of an edge $d$ of $\partial \Delta$. Then, the above property of $\Xi$ with respect to $\partial \Delta$ ensures that the circle $p \times C_f$ maps one to one to the circle $\mu^{-1}(p)$ under the quotient $d \times T \rightarrow d \times T/(M^{d}_{\R})^{\perp}$. This implies that $\mu^{-1}(p) \subset \mathcal L$ and that it is a smooth boundary component. 

If $p \in \partial \Delta \cap \Xi$ is a vertex of $\Delta$, then the local model for a tropical curve ending on $p$ and satisfying (P) is 
\[ \begin{split}
         & \Delta = \R_{\geq 0} \times \R_{\geq 0}, \\
         & X_{\Delta} = \C^2, \quad \mu(z_1, z_2 ) = (|z_1|^2, |z_2|^2)           \\
         & \Xi = \{ x_1 -x_2 = 0 \}. 
   \end{split} 
\]
Then it is easy to see that the lift of $\Xi$ is $\mathcal L = \{ z_1 = \bar z_2 \}$, which is smooth. 
\end{proof}

\begin{ex} \label{tori_P2} Let $\Delta$ and $\Xi$ be as in Figure \ref{co_am_torus_p2}. 
\begin{figure}[!ht] 
\begin{center}
\includegraphics{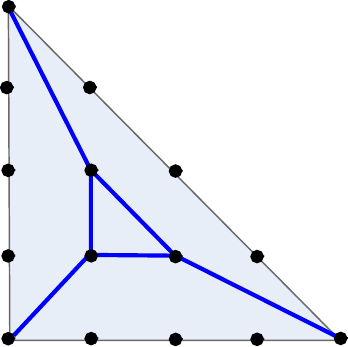}
\caption{} \label{co_am_torus_p2}
\end{center}
\end{figure}
Then $X_{\Delta} = \PP^2$ and $\Xi$ satisfies the property stated above. In particular its lift $\mathcal L$ does not have a boundary, therefore it is a smooth Lagrangian torus in $\PP^2$. The tropical curve comes naturally in a one parameter family, where the parameter is the size of the triangle defined by $\Xi$. Or, even better, the parameter is the area of the holomorphic disc with boundary on $\mathcal L$ which can be constructed over a segment joining an edge of $\Delta$ and the corresponding parallel edge of the triangle. Since this area decreases as we increase the size of the triangle, this implies that elements of this family are in different Hamiltonian classes. Notice also that the same tropical curve can be lifted in different ways, by twisting the lifts of the edges as explained in \S \ref{lifts_same_crv}, therefore we also have the Lagrangian tori $\mathcal L_{\sigma}$ for each choice of twist $\sigma$. Of course all these tori represent the trivial homology class, but we do not know if they represent different Hamiltonian classes. Moreover, what is the relation between these tori and other well known Lagrangian tori in $\PP^2$, e.g. the fibres of the moment map? 
\end{ex}

\begin{ex} \label{tori_P1P1} Let $\Delta$ and $\Xi$ be as in Figure \ref{co_am_torus_p1p1}, then $X_{\Delta} = \PP^1 \times \PP^1$ and $\Xi$ satisfies property (P). 
\begin{figure}[!ht] 
\begin{center}
\includegraphics{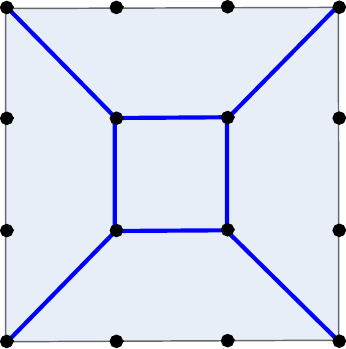}
\caption{} \label{co_am_torus_p1p1}
\end{center}
\end{figure}
Therefore the lift $\mathcal L$ is a Lagrangian torus. Also in this case $\Xi$ comes in a one parameter family, giving rise to a one parameter family of Lagrangian tori which are not Hamiltonian equivalent. Moreover we have the twists $\mathcal L_{\sigma}$. 
\end{ex}

\subsection{Non-orientable Lagrangian surfaces} \label{non_orientable_ex}
It is possible to construct non-orientable Lagrangian surfaces in $\C^2$ as lifts of tropical curves in $\Delta = \R_{\geq 0} \times \R_{\geq 0}$. This idea was explained to us by Mikhalkin (see \cite{mikh_trop_to_lag}). 
\begin{ex}
Consider the tropical curve $\Xi$ in $\Delta = \R_{\geq 0} \times \R_{\geq 0}$ depicted in Figure \ref{non_orientable}.
\begin{figure}[!ht] 
\begin{center}
\includegraphics{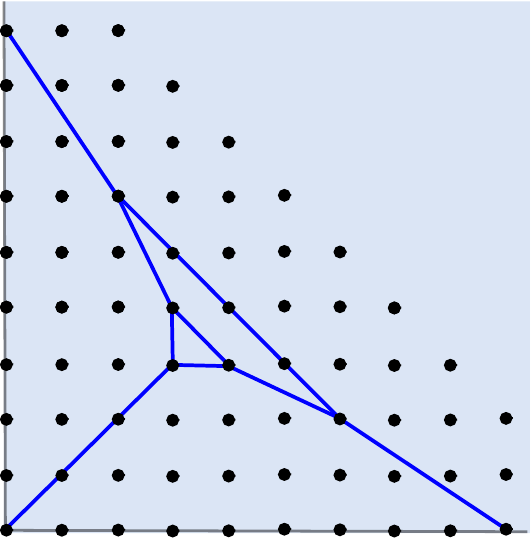}
\caption{} \label{non_orientable}
\end{center}
\end{figure}
The interesting fact is that at the two points where it hits the two edges of $\Delta$, it does not satisfy property (P). In fact at these points the primitive integral tangent vectors to the edge $\check f$ of $\Xi$ and to the edge $d$ of $\Delta$ span a lattice of index two. This implies that the circle $p \times C_f$, which is the intersection of $\mathcal L^{\infty}$ with $d \times T$, maps two to one to the circle $\mu^{-1}(p)$ under the quotient $d \times T \rightarrow d \times T/(M^{d}_{\R})^{\perp}$. In particular the lift $\mathcal L$ is smooth without boundary but not orientable. Indeed in this example $\mathcal L$ is homeomorphic to the connected sum of a genus two surface with two copies of $\RP^2$, it thus has Euler characteristic $-4$. The first examples of non-orientable Lagrangian surfaces in $\C^2$ were given by Givental, \cite{givental_Lag_non_orient} who constructed examples of Euler characteristic $-4k$ for $k\geq 1$. Audin \cite{audin_lag_givental} proves that the Euler characteristic of a non orientable Lagrangian surface in $\C^2$ must be divisible by 4. More recently  Nemirovski \cite{nemirovski_lag_Kleinbot} proved that there exists no Lagrangian embedding of the Klein bottle in $\C^2$. 
\end{ex}

\subsection{Monotone Lagrangian tori}
We give two examples of Lagrangian monotone tori, one in $\PP^2$ and one $\PP^1 \times \PP^1$. These are limit cases of Examples \ref{tori_P2} and \ref{tori_P1P1}.

\begin{ex}
Let $\Delta$ and $\Xi$ be as in Figure \ref{co_am_genus1_p2}.  Then $X_{\Delta}= \PP^2$. 
\begin{figure}[!ht] 
\begin{center}
\includegraphics{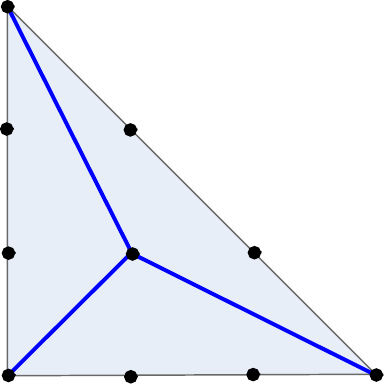}
\caption{} \label{co_am_genus1_p2}
\end{center}
\end{figure}
Notice that $\Xi = \Delta \cap \Xi^{\infty}$, where $\Xi^{\infty}$ is as in Figure \ref{co_am_genus1}, in  particular $\Xi$ is not smooth. Nevertheless we showed in Example \ref{non_smooth_trivalent} that we can construct a Lagrangian lift of $\Xi^{\infty}$ as the graph of the differential of a function $F$ defined on the coamoeba of Figure \ref{co_am_genus1}. As in the proof of Theorem \ref{trop_to_lag}, we can deform this graph so that over each edge $\check f$, just away from the vertex, the lift actually coincides with the cylinder $\check f \times C_f$. In particular we can construct a smooth Lagrangian torus $\mathcal L$ in $\PP^2$, as in Theorem \ref{lagran_bnd}. Notice that we can view this example as a limit case of Example \ref{tori_P2}, where the triangle shrinks down to a point.

\begin{thm} The Lagrangian torus $\mathcal L$ in $\PP^2$ lifting the tropical curve in Figure \ref{co_am_genus1_p2} is monotone. 
\end{thm}

\begin{proof} Recall that in the definition of monotone Lagrangian submanifold we have two linear maps defined on $H_2(\PP^2, \mathcal L)$: the area $\omega: \alpha \mapsto \int_{\alpha} \omega$ and the Maslov index $\mu$. We must prove that these two maps are proportional. A basis of $H_{2}(\PP^2, \mathcal L) \cong \Z^3$ is given by a generator $\tau$ of $H_2(\PP^2)$ and by the classes of two disks with boundaries the $1$-cycles in $\mathcal L$ depicted in Figure \ref{co_am_genus1_cycles}. To be more precise the cycles are the images of the two curves depicted in Figure \ref{co_am_genus1_cycles} via the graph of the differential of the function $F$ on the coamoeba. 
\begin{figure}[!ht] 
\begin{center}
\includegraphics{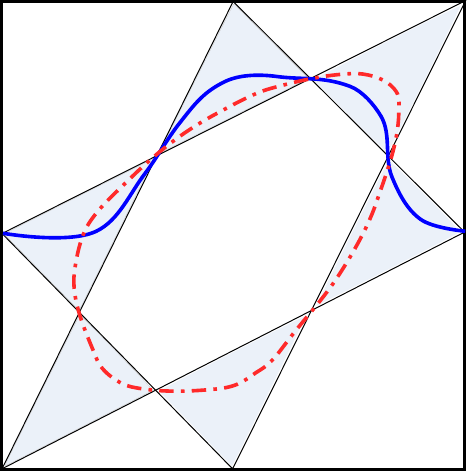}
\caption{Two $1$-cycles in $\mathcal L$ which are the boundaries of two disks in $H_2(\PP^2, \mathcal L)$ } \label{co_am_genus1_cycles}
\end{center}
\end{figure}
We have that 
\[ \mu(\tau) = 2 c_1(\tau) = 6 \]
(see for example Oh \cite{oh_floerCoh_psholdiskI}), while 
\[ \omega(\tau) = 3. \]
Now, let us recall Lemma 3.1 in \cite{Aroux-SYZ} which says that if $L$ is a special Lagrangian submanifold (or  with vanishing Maslov class) inside the complement of an anticanonical divisor $D$ of a K\"alher manifold $X$, then the Maslov index of a disk with boundary on $L$ is $2 \beta \cdot [D]$, i.e. twice the intersection number between the disk and the divisor. In our case $D$ is the toric boundary and we can use this result since $\mathcal L$ has vanishing Maslov class (see Proposition \ref{maslov_pants}). It is not hard to see that both the Maslov index and the area of disks can be computed by pretending that the two $1$-cycles are precisely the ones depicted in Figure \ref{co_am_genus1_cycles} (i.e. they are contained inside the torus fibre over the vertex of $\Xi$), and not the images of those curves via the graph of $dF$. Now, the vertex of $\Xi$ is the barycenter of $\Delta$ and thus the torus fibre over it is the Clifford torus, which is monotone. In particular if $\beta$ is a disk with boundary on the curve drawn with a continuous (blue) line, then
\[ \mu(\beta) = 2 \beta \cdot [D] =  2 \]
and
\[ \omega(\beta) = 1. \]
If $\beta$ is a disk with boundary on the dashed (red) line, then this disk can be taken inside the torus fibre thus
\[ \mu(\beta) = \omega(\beta) = 0. \]
These equalities show that $\mu = 2 \omega$ and thus that $\mathcal L$ is monotone.
\end{proof}
\end{ex}

\begin{ex} Let $\Delta$ and $\Xi$ be as in  Figure \ref{monotone_p1p1}. Then $X_{\Delta} = \PP^1 \times \PP^1$. Here $\Xi = \Delta \cap \Xi^{\infty}$ where, $\Xi^{\infty}$ is a in Example \ref{four_val}, therefore a lift $\mathcal L$ of $\Xi$ inside $X_{\Delta}$ can be constructed from a lift of $\Xi^{\infty}$. 
\begin{figure}[!ht] 
\begin{center}
\includegraphics{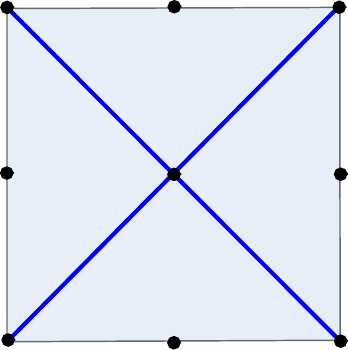}
\caption{} \label{monotone_p1p1}
\end{center}
\end{figure}
The coamoeba is as in Figure \ref{co_am_4valent}. Thus $\mathcal L$ is a Lagrangian torus. Just as in the previous example, one can show that $\mathcal L$ is monotone.
\end{ex}

We wonder what is the relationship between these examples of monotone Lagrangian tori and other known examples, such as the Clifford torus (i.e. the torus fibre over the barycenter of $\Delta$) and the ones in \cite{chekanov_schlenk_tori}, \cite{vianna_infte_montone_tori}, \cite{abreu_gadbled_monotone_lag}. 

\section{Appendix}

Here we sketch a proof of the following

\begin{prop} \label{imh23} For $n=1$ and $2$, the map $\h$ constructed in \S \ref{theconstr}, restricted to $\inter{C}^+$, is a local diffeomorphism .

\end{prop} 

Unfortunately we are not yet able to provide a rigorous proof for all values of $n$.   Obviously $\h$ is a local diffeomorphism if and only if the determinant of the Hessian of $F$ is nowhere vanishing inside $\inter C$. Proving the latter, via a direct computation, is difficult, so we will follow a different method. The idea is the following. Fix a point $y \in \inter C^+$ and for $j=0, \ldots, n$, let $\sigma_j$ be the straight line passing through $y$ and the vertex $p_j$ and let $v_j$ be a tangent vector of $\sigma_j$ at $y$. 
We will show that the images of the vectors $v_0, \ldots, v_n$ via the differential of $\h$ at $y$ are linearly independent. The advantage of this approach is that the symmetries of $\h$ studied in \S \ref{symmetries} exchange lines emanating from one vertex with lines emanating from another one. Therefore it is enough to study the image, via $\h$, of lines through $p_0$. 

Given equation \eqref{fy}, we can assume the components of $\h$ are
\[ h_j(y) = \frac{\cos \left( 2y_j + \sum_{k \neq j} y_k \right) \prod_{k \neq j} \sin y_k}{\left[ \cos \left( \sum_{k=1}^{n+1}y_k \right) \prod_{k=1}^{n+1} \sin y_k \right]^{\frac{n}{n+1}}}, \] 
where, for simplicity, we removed the factor $\frac{1}{n+1}$. 

The line $\sigma_0$ through $p_0$ and $y$ is of the type
\[ \sigma_0(t) = (a_1t, \ldots, a_{n+1}t)\]
with 
\begin{equation} \label{aj}
    \sum_{j=1}^{n+1} a_j = \frac{\pi}{2}
\end{equation}
We have $\sigma_0(t) \in \inter C$ if and only if $0 < t < 1$. 
Let $\gamma_0(t) = \h(\sigma_0(t))$. Then $\gamma_0= (\gamma_{0,1}, \ldots, \gamma_{0,n+1})$ with
\[ \gamma_{0,j}(t) = \frac{\cos\left( a_j + \frac{\pi}{2}\right)t  \prod_{k \neq j} \sin a_k t}{\left[ \cos \frac{\pi}{2}t \, \prod_{k=1}^{n+1} \sin a_k t \right]^{\frac{n}{n+1}}}.\]

\begin{lem} \label{tangent_ineq} If the $a_j$'s satisfy \eqref{aj} and 
\begin{equation} \label{orderaj}
    0 < a_{n+1} \leq a_n \leq \ldots \leq a_{1}
\end{equation}
then $\gamma$ satisfies
\[ \gamma'_{0,n+1}  \leq \gamma'_{0,n} \leq  \ldots \leq \gamma'_{0,1} < 0 \]
\end{lem}

\begin{proof} Observe that $\gamma_{0,j}(t) \geq 0$ if and only if $t \in \left(0, \frac{\pi}{\pi + 2a_j} \right)$. Therefore it follows from \eqref{orderaj} that $\gamma_{0,1}$ is the first to become negative followed by $\gamma_{0,2}$ and so on. Now let
\[ \rho_k(t) = \frac{\cos\left( a_k + \frac{\pi}{2}\right)t }{\left[ \cos \frac{\pi}{2}t \right]^{\frac{n}{n+1}}} \quad \text{and} \quad \beta_{j,k}(t) = \frac{ \sin a_j t}{\sin a_k t} \]
Notice that $\beta_{j,k} > 0$ for all $t \in (0,1)$ and that $\rho_k$ has the same sign as $\gamma_{0,k}$.
We have that 
\[ \gamma_k = \rho_k \left(  \prod_{j \neq k} \beta_{j,k} \right)^{\frac{1}{n+1}} \]
One can show that for all $t \in (0,1)$ 
\[\rho_k' <0  \quad \text{and} \quad \beta'_{j,k} \leq 0 \ \ \text{if} \ a_k \leq a_j. \]
In particular \eqref{orderaj} implies that 
\[ \gamma'_{0,n+1}(t) < 0 \quad \forall t \in \left( \left. 0,  \frac{\pi}{\pi+2 a_{n+1}} \right. \right] \]
and
\[ \gamma'_{0,1}(t) < 0 \quad \forall t \in \left[ \left.  \frac{\pi}{\pi+2 a_{1}}, 1 \right) \right. \]
We have that
\[ \frac{\gamma_{0,k}}{\gamma_{0,j}} = \frac{\cos \left(a_k+\frac{\pi}{2} \right)t \, \sin a_j t }{\cos \left(a_j+\frac{\pi}{2} \right)t \, \sin a_k t}. \]
One can show that 
\[ \left( \frac{\gamma_{0,k}}{\gamma_{0,j}} \right)' \geq 0 \quad \text{on} \ (0,1) \  \text{if} \ a_k \leq a_j. \]
In particular 
\[ \left( \frac{\gamma_{0,n+1}}{\gamma_{0,1}} \right)' \geq 0 \quad \text{on} \ (0,1). \]
Now, since
\[ \gamma'_{0,n+1} = \frac{1}{\gamma_{0,1}}\left(  \left( \frac{\gamma_{0,n+1}}{\gamma_{0,1}} \right)' \gamma_{0,1}^2 + \gamma'_{0,1} \gamma_{0,n+1} \right), \]
we have that all of the above implies that 
\[ \gamma'_{0,n+1} < 0 \quad  \forall t \in \left(   \frac{\pi}{\pi+2 a_{n+1}}, 1 \right). \]
Similarly
\[ \gamma'_{0,1} = \frac{1}{\gamma_{0,n+1}}\left(  \gamma'_{0,n+1} \gamma_{0,1} - \left( \frac{\gamma_{0,n+1}}{\gamma_{0,1}} \right)' \gamma_{0,1}^2  \right), \]
which implies
\[ \gamma'_{0,1}(t) < 0 \quad \forall t \in \left( 0, \frac{\pi}{\pi+2 a_{1}} \right). \]
This shows that both $\gamma'_{0,1}$ and $\gamma'_{0,n+1}$ are negative. A similar argument shows that also $\gamma'_{0,k}$ is negative for all the other $k$'s.

Let us now prove that $\gamma'_{0,k+1} - \gamma'_{0,k} \leq 0$. We have that 
\[ \gamma_{0,k+1}(t) - \gamma_{0,k}(t) =  \frac{\cos \frac{\pi}{2} t \, \sin\left( a_k- a_{k+1} \right)t  \prod_{j \neq k,k+1} \sin a_j t}{\left[ \cos \frac{\pi}{2}t \, \prod_{k=1}^{n+1} \sin a_k t \right]^{\frac{n}{n+1}}}.\]
One can show that the derivative of this function is negative.
\end{proof}

Let us use this lemma to prove that $\h$ is a local diffeomorphism in the case $n=1$. Using the symmetries we can assume that $y \in \mathcal W^+_{12} \cap \mathcal W^+_{2}$, i.e. that $y$ satisfies $0 < y_2 \leq y_1$ and $2y_1 +y_2 \leq \pi/2$. Then consider the lines $\sigma_0$, passing through $p_0$ and $y$, and $\sigma_1$, passing through $p_1$ and $y$. Let $\gamma_{j} = \h \circ \sigma_j$. See Figure \ref{curves_in_amoeba} for a plot of the curves $\sigma_j$ and $\gamma_j$

Then, using the above lemma for $\gamma_0$ we have 
\[ \gamma'_{0,2} \leq \gamma'_{0,1} < 0. \]
To estimate the derivatives of $\gamma_1$, notice that $\gamma_1$ is the image, via the transformation $T^*_1$ defined in \S \ref{symmetries}, of a curve through $\tilde y=T^{*}_1y$ and $p_0$, moreover $\tilde y$ satisfies $0 < \tilde y_2 \leq \tilde y_1$.  Using Lemma \ref{eq_action} and the lemma above we have that the derivatives of $\gamma_1$ satisfy
\[ \gamma'_{1,2} \leq 0 < \gamma'_{1,1} \]
It is easy to see then that $\gamma'_0$ and $\gamma'_{1}$ must be linearly independent, i.e. the differential of $\h$ is invertible at $y$. 
\begin{figure}[!ht] 
\begin{center}
\scalebox{.7}{\includegraphics{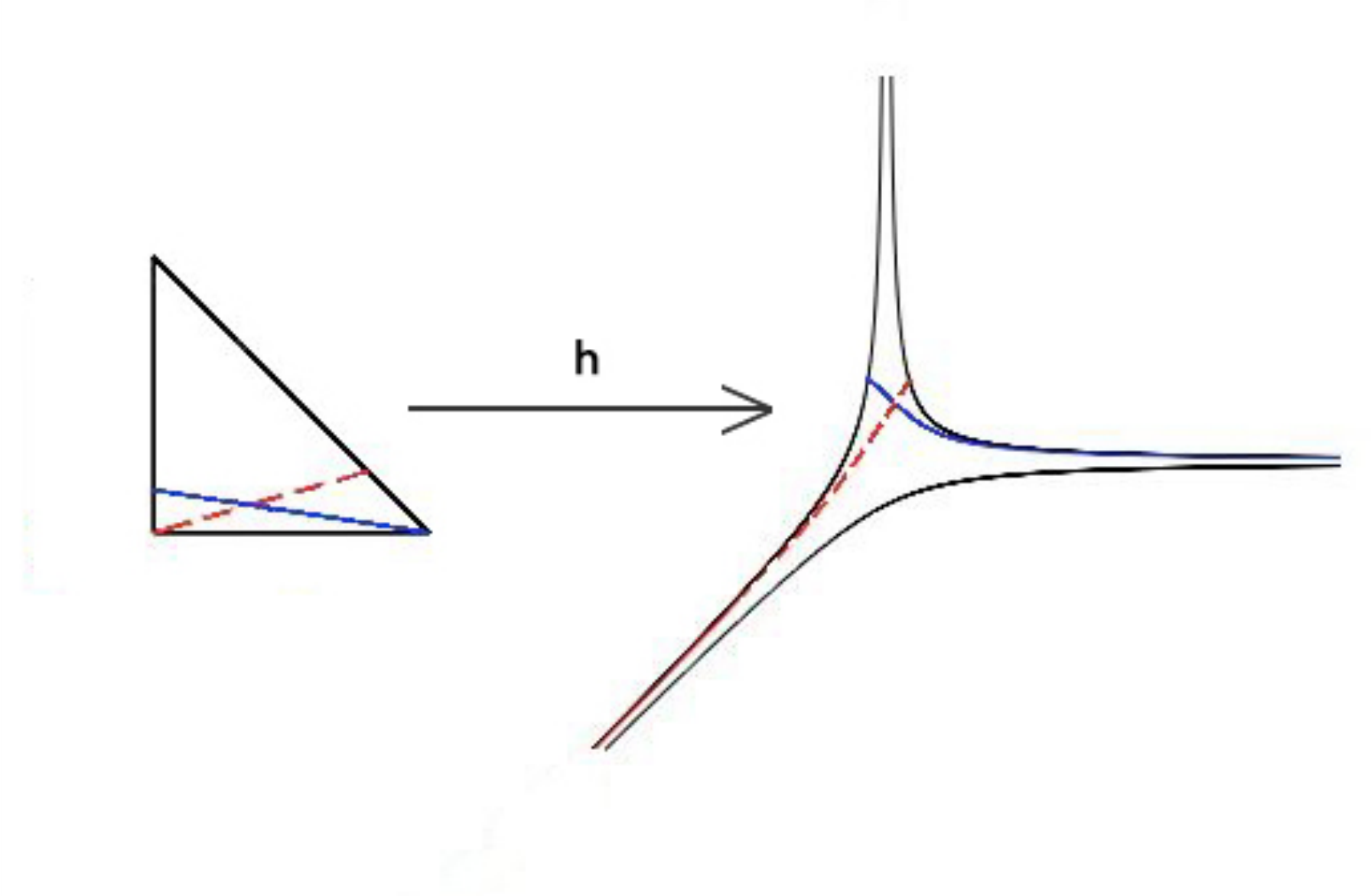}}
\caption{The curves $\sigma_j$ and $\gamma_j$. Dashed (red) lines for $j=0$ and continuous (blue) for $j=1$.} \label{curves_in_amoeba}
\end{center}
\end{figure}

For the case $n=2$, we need to consider another type of curve. Let $\tau$ be a line segment from a point on the edge $E_{12}$ to a point on the egde $E_{03}$, i.e. given $a, b \in (0, \pi/2)$,
\begin{equation} \label{tao}
        \tau(t) = \left( \left(\frac{\pi}{2}-b\right)t, bt, (1-t)a \right).
\end{equation}
We have $\tau(0)=(0,0,a) \in E_{12}$ and $\tau(1) = \left( \frac{\pi}{2}-b, b, 0 \right) \in E_{03}$. Now let 
\[ \eta = \h \circ \tau. \]

\begin{lem} \label{eta_deriv}
If $a,b \in (0, \pi/4)$, then the derivative $\eta'$ of $\eta$ satisfies
\[ \eta'_3 > 0 \quad \text{and} \quad \eta'_j < 0 \ \text{for} \ j=1,2. \]
\end{lem}

\begin{proof} The components of $\eta$ are
\[ \begin{split} 
    \eta_1(t) & = \frac{\cos \left( \left( \pi - a-b\right)t + a\right) \, \sin bt \, \sin a(1-t)}{\left( \cos \left( \left( \frac{\pi}{2}-a \right)t+ a \right) \, \sin \left(\frac{\pi}{2}-b \right)t\, \sin bt \, \sin a(1-t) \right)^{\frac{2}{3}}}, \\
    \eta_2(t) & = \frac{\cos \left( \left( \frac{\pi}{2} - a+b\right)t + a\right) \, \sin \left(\frac{\pi}{2}-b \right)t \, \sin a(1-t)}{\left( \cos \left( \left( \frac{\pi}{2}-a \right)t+ a \right) \, \sin \left(\frac{\pi}{2}-b \right)t\, \sin bt \, \sin a(1-t) \right)^{\frac{2}{3}}}, \\
    \eta_3(t) & = \frac{\cos \left( \left( \frac{\pi}{2} - 2a\right)t + 2a\right) \, \sin \left(\frac{\pi}{2}-b \right)t \, \sin bt}{\left( \cos \left( \left( \frac{\pi}{2}-a \right)t+ a \right) \, \sin \left(\frac{\pi}{2}-b \right)t\, \sin bt \, \sin a(1-t) \right)^{\frac{2}{3}}}.
\end{split}
\]
and the statement of the Lemma can be verified by direct computations.
\end{proof}

Now assume, without loss of generality, that $y$ satisfies $0 < y_3 < y_2 < y_1$ and $2y_1 + y_2 + y_3 \leq \frac{\pi}{2}$. Consider the lines $\sigma_0$, $\sigma_1$ and $\sigma_2$ respectively joining $p_0$, $p_1$ and $p_2$ to $y$. Assume that they are oriented so that they point from $p_j$ to $y$. As above, let $\gamma_j = \h \circ \sigma_j$. From Lemma \ref{tangent_ineq} it follows that the derivatives of $\gamma_0$ satisfy
\begin{equation} \label{gamma0}
  \gamma'_{0,3} \leq \gamma'_{0,2} \leq \gamma'_{0,1} < 0.
\end{equation}
Using the symmetries we can also show that the derivatives of $\gamma_1$ satisfy
\begin{equation} \label{gamma1}
    \gamma'_{1,3} \leq \gamma'_{1,2} \leq 0 < \gamma'_{1,1}
\end{equation}
and those of $\gamma_2$ satisfy
\begin{equation} \label{gamma2}  
    \gamma'_{2,3} \leq \gamma'_{2,1} \leq 0 < \gamma'_{2,2}.
\end{equation}
Now let $A$ be the matrix whose columns are $\gamma'_0$, $\gamma'_1$ and $\gamma'_2$, then 
\[ \det A = \gamma'_{0,3}(\gamma'_{1,1} \gamma'_{2,2}- \gamma'_{1,2}\gamma'_{2,1})- \gamma'_{0,1} \gamma'_{1,3} \gamma'_{2,2}+\gamma'_{0,2} \gamma'_{1,3} \gamma'_{2,1}+\gamma'_{0,1} \gamma'_{1,2} \gamma'_{2,3}-\gamma'_{0,2} \gamma'_{1,1} \gamma'_{2,3} \]
\begin{lem}
We have that 
\[ \gamma'_{1,1} \gamma'_{2,2}- \gamma'_{1,2}\gamma'_{2,1} > 0. \]
\end{lem}
\begin{proof} Let us consider the plane containing $\sigma_1$ and $\sigma_2$ and let $q$ be the point of intersection between this plane and the edge $E_{12}$. Now let $\tau$ be the line passing through $q$ and $y$ and oriented so that it points from $q$ to $y$. This line is of type \eqref{tao}, with $a,b \in (0, \pi/4)$, therefore if $\eta = \h \circ \tau$, the derivatives of $\eta$ satisfy Lemma \ref{eta_deriv}. This, together with \eqref{gamma1}, implies that 
\[ \eta'_1 \gamma'_{1,2}- \eta_2' \gamma'_{1,1} > 0. \]
This inequality implies that if we consider the map from the triangle with vertices $q, p_1, p_2$ to $\R^2$ given by $\h$ followed by projection onto the first two coordinates, then the differential of this map is invertible at $y$. Now observe that the pairs of tangent vectors at $y$ given by $\{ \sigma'_1, \sigma'_2 \}$ and $\{ \tau', \sigma'_1 \}$ define the same orientation on this triangle. This  proves the lemma. 
\end{proof}
It is now easy to see that this lemma together with \eqref{gamma0}, \eqref{gamma1} and \eqref{gamma2} implies that $\det A <0$ and hence that the differential of $\h$ is invertible. 

\begin{cor} \label{HessFneg} Assuming $n=1$ or $2$. Let $F$ be as in \eqref{Fglob}, then the Hessian of $F$, restricted to $\inter C^+$, is negative definite.
\end{cor}
\begin{proof} Since $F$ is positive on $C^+$ and vanishes on its boundary, it must have an interior maximum. At this point the Hessian is negative definite, but since the determinant of the Hessian never vanishes, it must be negative definite everywhere.
\end{proof}
\ \bibliographystyle{plain}

\vspace{1cm}
\begin{flushleft}
Diego MATESSI \\
Dipartimento di Matematica \\
Universit\`a degli Studi di Milano \\
Via Saldini 50 \\
I-20133 Milano, Italy \\
E-mail address: \email{diego.matessi@unimi.it}
\end{flushleft}

\end{document}